\documentclass[11pt,leqno]{amsart}
\usepackage{amsthm,amsfonts,amssymb,amsmath,oldgerm}
\numberwithin{equation}{section}
\usepackage[utf8]{inputenc}
\usepackage[T1]{fontenc}
\usepackage[english]{babel}
\usepackage{fullpage}
\usepackage{amsmath}
\usepackage{amsfonts}
\usepackage{amssymb}
\usepackage{setspace}
\usepackage{wrapfig}
\usepackage{stmaryrd}
\usepackage{mathrsfs}
\usepackage{fancyhdr}
\usepackage{graphicx}
\usepackage{mathabx}
\usepackage{tikz}
\usetikzlibrary{patterns,decorations.pathreplacing}

\usepackage{psfrag}
\usepackage{listings} 

\usepackage[scale = .72, a4paper]{geometry}

\renewcommand\d{\partial}
\newcommand\dD{\textrm{d}}
\newcommand\eD{\textrm{e}}

\def\eps{\varepsilon }

\newcommand{\I}{{\rm I}}
\newcommand{\funcun}{{\rm \bf 1}}

\newcommand\br{\begin{remark}}
\newcommand\er{\end{remark}}
\newcommand\bp{\begin{pmatrix}}
\newcommand\ep{\end{pmatrix}}
\newcommand{\be}{\begin{equation}}
\newcommand{\ee}{\end{equation}}
\newcommand\ba{\begin{equation}\begin{aligned}}
\newcommand\ea{\end{aligned}\end{equation}}
\newcommand\ds{\displaystyle}

\newcommand{\beg}{\begin{example}}
\newcommand{\eeg}{\end{exaplem}}
\newcommand{\bpr}{\begin{proposition}}
\newcommand{\epr}{\end{proposition}}
\newcommand{\bt}{\begin{theorem}}
\newcommand{\et}{\end{theorem}}
\newcommand{\bc}{\begin{corollary}}
\newcommand{\ec}{\end{corollary}}
\newcommand{\bl}{\begin{lemma}}
\newcommand{\el}{\end{lemma}}
\newcommand{\bd}{\begin{definition}}
\newcommand{\ed}{\end{definition}}
\newcommand{\brs}{\begin{remarks}}
\newcommand{\ers}{\end{remarks}}

\newtheorem{theorem}{Theorem}[section]
\newtheorem{proposition}[theorem]{Proposition}
\newtheorem{corollary}[theorem]{Corollary}
\newtheorem{lemma}[theorem]{Lemma}

\theoremstyle{remark}
\newtheorem{remark}[theorem]{Remark}
\theoremstyle{definition}
\newtheorem{definition}[theorem]{Definition}

\newtheorem{example}[theorem]{Example}

\newcommand\R{\mathbf R}

\newcommand\T{\mathbf T}
\newcommand{\N}{\mathbf N}
\newcommand{\Z}{\mathbf Z}

\newcommand{\lla}{\left\langle}
\newcommand{\rra}{\right\rangle}
\newcommand{\nrm}{\vvvert}

\newcommand{\RR}{{\mathbf R}}
\newcommand{\TT}{{\mathbf T}}

\newcommand{\RM}{{\mathbf R}}

\newcommand{\TM}{{\mathbf T}}

\newcommand\cE{{\mathcal E}}
\newcommand\cF{{\mathcal F}}

\newcommand\cH{{\mathcal H}}

\newcommand\cR{{\mathcal R}}

\newcommand\tf{{\widetilde f}}
\newcommand\tg{{\widetilde g}}

\newcommand\tn{{\widetilde n}}

\newcommand\teta{{\widetilde \eta}}
\newcommand\ttheta{{\widetilde \theta}}

\newcommand\tE{\widetilde{E}}

\newcommand\tQ{\widetilde{Q}}

\newcommand\tT{\widetilde{T}}

\newcommand\phiinf{\phi_\infty^\delta}
\newcommand\Einf{E_\infty^\delta}
\newcommand\rhoinf{\rho_\infty}
\newcommand\finf{f_\infty}

\newcommand{\tref}{t_{\textrm{ref}}}
\newcommand{\fref}{f_{\textrm{ref}}}
\newcommand{\fFP}{f_{\textrm{FP}}}
\newcommand{\Eref}{E_{\textrm{ref}}}
\newcommand{\tEref}{\tE_{\textrm{ref}}}
\newcommand{\rhoref}{\rho_{\textrm{ref}}}
\newcommand{\nref}{n_{\textrm{ref}}}
\newcommand{\tnref}{\tn_{\textrm{ref}}}
\newcommand{\jref}{j_{\textrm{ref}}}
\newcommand{\Sref}{\mathbb{S}_{\textrm{ref}}}

\newcommand{\href}{h_{\textrm{ref}}}
\newcommand{\hlin}{h_{\textrm{lin}}}
\newcommand{\has}{h_{\textrm{as}}}
\newcommand{\nas}{n_{\textrm{as}}}
\newcommand{\rhoas}{\rho_{\textrm{as}}}
\newcommand{\Eas}{E_{\textrm{as}}}

\title{
Large-time behavior of solutions to Vlasov-Poisson-Fokker-Planck equations: from 
evanescent collisions to diffusive limit
}

\author{Maxime Herda}
\address{Maxime Herda,
Universit\'e de Lyon,
CNRS UMR 5208,
Universit\'e Lyon 1,
Institut Camille Jordan,
43 bd 11 novembre 1918;
F-69622 Villeurbanne Cedex, FRANCE}
\email{{\tt Herda@math.univ-lyon1.fr}}

\author{L.Miguel Rodrigues}
\address{Luis Miguel Rodrigues,
Universit\'e de Rennes 1,
IRMAR, UMR CNRS 6625,
263 avenue du General Leclerc;
F-35042 Rennes Cedex, FRANCE}
\email{{\tt luis-miguel.rodrigues@univ-rennes1.fr}}
\thanks{Research of L.Miguel Rodrigues was partially supported by the ANR project
BoND ANR-13-BS01-0009-01.\\
}

\begin{document}

\begin{abstract}
The present contribution investigates the dynamics generated by the two-dimensional Vlasov-Poisson-Fokker-Planck equation for charged particles in a steady inhomogeneous background of opposite charges. We provide global in time estimates that are uniform with respect to initial data taken in a bounded set of a weighted $L^2$ space, and where dependencies on the mean-free path $\tau$ and the Debye length $\delta$ are made explicit. In our analysis the mean free path covers the full range of possible values: from the regime of evanescent collisions $\tau\to\infty$ to the strongly collisional regime $\tau\to0$. 
As a counterpart, the largeness of the Debye length, that enforces a weakly nonlinear regime, is used to close our nonlinear estimates. Accordingly we 
pay a special attention to relax as much as possible the $\tau$-dependent constraint on $\delta$ ensuring exponential decay with explicit $\tau$-dependent rates towards the stationary solution. 
In the strongly collisional limit $\tau\to0$, we also examine all possible asymptotic regimes selected by a choice of observation time scale. Here also, our emphasis is on strong convergence, uniformity with respect to time and to initial data in bounded sets of a $L^2$ space. 
Our proofs rely on a 
detailed 
study of the nonlinear elliptic equation defining stationary solutions and a careful tracking and optimization of parameter dependencies of hypocoercive/hypoelliptic estimates.
\end{abstract}

\date{\today}
\maketitle



\tableofcontents


\section{Introduction}\label{s:introduction}

In a periodic box $\TM^2\,=\,\R^2/\Z^2$, we consider a large number of charged particles subject to self-consistent electrostatic forces and interacting with a fixed background of steady heavy particles with opposite charge as well as a thermal bath. The system can be described at a kinetic level by a distribution function $f:\RM_+\times\TM^2\times\RM^2 \rightarrow \RM_+$, $(t,x,v)\mapsto f(t,x,v)$ obeying the Vlasov-Poisson-Fokker-Planck (VPFP) system which reads in dimensionless form
 \be
  \left\{
 \begin{array}{l}
\ds\partial_{t}f + v\cdot\nabla_{x}f  - \nabla_{x}\phi \cdot \nabla_{v}f\ =\ \frac{1}{\tau}\,\textrm{div}_{v} \left( vf + \nabla_{v}f\right),\\
\,\\
\ds-\delta^2\Delta_{x}\phi\ =\ \rho - \rho_*,\qquad \rho\ =\ \int_{\R^2}\,f\,\dD{v}
\end{array}
\right.
\label{VPFPf}
 \ee
and is completed with the prescription of an initial data, $f(0,\cdot,\cdot)=f_0$.
 
The number density of background particles is given as $\rho_*:\TM^2\to\R_+$ whereas $\phi:\R_+\times\T^2\to\R$ is the electrostatic potential satisfying a Poisson equation.
Interactions with the thermal bath are modeled by the Fokker-Planck term at the right-hand side of  the first equation. The characteristic temperature of the bath is scaled to $1$. The respective total charges have also been scaled to $1$,
$$
\int_{\T^2\times\R^2}f_0\,=\,1\,,\qquad\qquad \int_{\T^2}\rho_*\,=\,1\,,
$$
the constraint being preserved by the time evolution. 

The parameter $\tau$ denotes the scaled mean free path 
between two "collisions" with the thermal bath. The scaled Debye length $\delta$ measures the radius of electrostatic influence of an isolated particle. The asymptotics $\delta\rightarrow 0$ is called the quasineutral limit and the terminology quasineutral parameter is also used for $\delta$. 
The quasineutral regime is very relevant for plasma dynamics. However, for collisionless equations and initial data with general shape profiles, instabilities are known to arise in the quasineutral limit \cite{han_2015_stability}. Since we consider general shapes and all Knudsen numbers uniformly in time, we need to remain far from this complex regime. In particular we allow ourselves to use the largeness of $\delta$ to close nonlinear estimates. While, as 
we discuss further later on, our strategy may potentially be extended 
to any dimension, 
we anticipate that 
the precise outcome would be significantly modified by another dimensional choice therefore we also choose to restrict ourselves to the two dimensional setting. 

The VPFP system, as expounded here or with some variants, including the consideration of  gravitational forces instead of electrostatic forces, has a long history. A derivation of the model and references to even earlier derivations may already be found in a seminal piece of work by Chandrasekhar \cite{chandrasekhar_1943_stochastic}. Concerning the Cauchy problem in two dimensions, first global well-posedness results were obtained by Neunzert, Pulvirenti and Triolo \cite{neunzert_1984_on} on $\R^2$ for bounded compactly-supported initial data, by the method of stochastic characteristics; then by Degond \cite{degond_1986_global} for a frictionless version of the system, on $\R^2$, for $W^{1,1}$ data with finite moments in velocity of more than second order, by relying mostly on suitable maximum principles. Since then those results have been extended and improved in various ways \cite{victory_1990_classical,carrillo_1995_initial,Ono-VPFP}.

Our goal is to provide a description of the dynamics of solutions to Equation~\eqref{VPFPf} on every possible time scale, from initial data to exponential convergence towards a stationary state, and in any regime of the collisional parameter $\tau$. We also aim at providing strong convergence results uniform with respect to initial data taken from bounded sets. Our normalization already contains $\|f_0\|_{L^1(\T^2\times\R^2)}=1$. However to benefit from a Hilbert structure we shall use a weighted $L^2$ space embedded in $L^1$ instead of $L^1$ itself. Namely we introduce the norm $\|\cdot\|_{L^2(M^{-1})}$ defined by
\[
 \|u\|_{L^2(M^{-1})}^2\ =\ \iint_{\TT^2\times\RR^2}\tfrac{|u(x,v)|^2}{M(v)}\,\dD x\,\dD v
\]
where $M$ is the local Maxwellian
\[
M(v) = \frac{1}{(2\pi)}\,e^{-\tfrac12|v|^2}\,.
\]
The main advantage in this particular weight choice stems from the fact that the Fokker-Planck operator is symmetric on $L^2(M^{-1})$. Then we choose $f_0\geq0$ such that 
$$
\|f_0\|_{L^2(M^{-1})}\leq R_0
\,,
$$
where $R_0\geq1$ is fixed but arbitrary and obtain bounds depending on $R_0$ but not on $f_0$ itself.

As a preliminary observation, note that stationary states $(\finf,\phiinf)$ are characterized by
$$
\left\{
\begin{array}{l}
\ds
v\cdot\nabla_{x}\finf  - \nabla_{x}\phiinf \cdot \nabla_{v}\finf\ =\ \frac{1}{\tau}\,\textrm{div}_{v} \left( v\finf + \nabla_{v}\finf\right)\,,\\
\ds
-\delta^2\Delta_{x}\phiinf\ =\ \rhoinf - \rho_*\,,
\qquad \rhoinf\ =\ \int_{\R^2}\,\finf\,\dD{v}.
\end{array}
\right.
$$
Since the Fokker-Planck operator is symmetric on $L^2(\eD^{-\phiinf}M^{-1})$ and the transport part is skew-symmetric on this space, if for instance $\phiinf$ is bounded and $\finf\in L^2(M^{-1})$ then the two parts must vanish separately when applied to $\finf$. Therefore $\finf$ is a global Maxwell-Boltzmann distribution
\be
 \finf(x,v) = M(v)e^{-\phiinf(x)},
\label{globmax}
\ee
and the stationary equations reduce to
\be
  -\delta^2\Delta_{x}\phiinf = e^{-\phiinf} - {\rho_*}\,.
  \label{PoissonEmdenIntro}
\ee
Concerning the latter, we observe that

\bl
For any $\rho_*\in H^{-1}(\T^2)$ such that $\int_{\T^2}\rho_*=1$, for any $\delta>0$, Equation~\eqref{PoissonEmdenIntro} possesses a unique weak solution $\phiinf\in H^1(\T^2)$ and moreover $\int_{\TT^2}e^{-\phiinf}=1$.
\el

Our analysis involves a more detailed study of Equation~\eqref{PoissonEmdenIntro}. Though we have not found in the literature directly applicable results, the first steps of our analysis of equilibrium states is however by now essentially standard \cite{Gogny-Lions_equilibre-plasmas,Dressler_steady-VFP,Dolbeault_stationary-states,glassey_1996_steady,dolbeault_1999_free,CDMS,Duan-Yang-Zhu_stationary_VPB,Duan-Yang_stability_VPB}. It is noteworthy that the existence of nontrivial stationary states holds for most generalizations of \eqref{VPFPf}. To prevent it, one essentially needs to look either at frictionless versions, where the Fokker-Planck operator is replaced with a Laplacian in velocity, or at equations on $\R^2$ with no background density $\rho_*\equiv0$ and no confining potential \cite{glassey_1996_steady,dolbeault_1999_free}. Correspondingly, in the former case, in various perturbative settings, self-similar decay to zero at algebraic rates has been proven \cite{Carrillo-Soler-Vazquez_asymptotic_self-similar_3D-VPFP,Carpio_long-time-VPFP,Ono-Strauss_regular-solutions_VPFP,Kagei-VPFP-invariant-manifolds}.

In \cite{bouchut_1995_on,dolbeault_1999_free}, Bouchut and Dolbeault prove that in $\R^d$, $d\geq3$ in presence of a confining potential, when $\rho_*\equiv0$, solutions $f$ starting from initial data with finite mass, finite energy, finite entropy and such that $\nabla_x\phi\in L_{loc}^\infty((0,\infty);L^\infty(\R^d))$, converge in $L^1(\R^{2d})$ to a global Maxwell-Boltzmann distribution as time goes to infinity. Their proof is likely to extend to the case under consideration and our focus is not in reproving such behavior but in providing uniform bounds in the same space for the solution and the initial data, and including exponential decay to equilibrium.

Obviously this encompasses asymptotic stability of equilibrium states. For the equation under consideration, Hwang and Jang \cite{hwang_2013_vlasov} have already proved some asymptotic stability results for the case where $\rho_*\equiv1$ so that $\phiinf\equiv0$, in a topology requiring $(f_0-\finf)/\sqrt{\finf}$ to be small in $H^3(\T^2\times\R^2)$ (instead of $L^2$). Note that condition $\phiinf\equiv0$ brings a lot of exceptional cancellations in the analysis and a considerable part of asymptotic stability results for similar equations have been so far focused essentially on this case. Our result contains asymptotic stability in $L^2$ with explicit dependence on parameters of exponential convergence rates, and arbitrary $\rho_*$, provided $\delta$ is sufficiently large. Though we have not investigated this track, we anticipate that our strategy of proof would also yield exponential asymptotic stability in $L^2$ for any fixed $\delta$ provided $\rho_*-1$ is sufficiently small, in the spirit of \cite{Duan-Yang_stability_VPB}. During the completion of the present contribution, also appeared by Bedrossian \cite{Bedrossian_VPFP} an extension of \cite{hwang_2013_vlasov}, still with $\rho_*\equiv1$ and high-regularity weighted $L^2$ spaces (but with algebraic weights instead of Gaussian ones), to a version of the system where collisions are more nonlinear and model self-collisions instead of collisions with a thermal bath. More importantly to us, holding $\delta$ fixed, Bedrossian provides a careful study of dependencies on the parameter $\tau$ in the limit $\tau\to\infty$, benefiting both from Landau damping for the limiting Vlasov-Poisson system \cite{MouhotVillani-damping,BedrossianMasmoudiMouhot-damping} and from mixing-enhanced dissipation. This extends to the full nonlinear regime the former linearized analysis by Tristani \cite{Tristani-damping-weakly-collisional} and is similar in spirit to \cite{BedrossianMasmoudiVicol-enhanced-dissipation-inviscid-damping,BedrossianGermainMasmoudi-3D-CouetteI,BedrossianGermainMasmoudi-3D-CouetteII,BedrossianVicolWang-2DCouette,BedrossianGermainMasmoudi-3DCouette} that build on inviscid damping for shear flows of the Euler system \cite{BedrossianMasmoudi-inviscid-damping-note,BedrossianMasmoudi-inviscid-damping}.

It is important to note however that our nonlinear parameter shall not be the distance of the initial data to some equilibrium, that is essentially arbitrary here, but the inverse of the quasineutral parameter. In particular even when $\rho_*\equiv1$ our set of initial data contains data that fail to satisfy Penrose stability criterion, that is known to play a crucial role for the Vlasov-Poisson system both in the large-time limit \cite{penrose_1960_electrostatic,guo_1995_nonlinear,MouhotVillani-damping,BedrossianMasmoudiMouhot-damping} and in the quasineutral limit \cite{han_2015_quasineutral}. Though the currently available results \cite{degond_1986_global} concerning the approximation of solutions to \eqref{VPFPf} by those of the Vlasov-Poisson system are not precise enough to justify relevant heuristic arguments, this strongly hints at the fact that one cannot benefit from mixing properties at least initially and that one needs to take $\delta\to\infty$ when $\tau\to\infty$.

Though they do not state it in this precise way, the recent analysis by H\'erau and Thomann \cite{herau_2016_global} proves precisely that for any fixed $(\tau,R_0)$ there indeed exists a $\delta_0(\tau,R_0)$ such that when $\delta>\delta_0(\tau,R_0)$ one may obtain global bounds for the solution, for the version of the system set on $\R^2$, with a confining potential but with $\rho_*\equiv0$. Our main goal here is to provide explicit upper bounds for those $\delta_0(\tau,R_0)$. The analysis of H\'erau and Thomann scales badly with respect to $\tau$, due to the anisotropic nature of hypocoercivity/hypoellipticity. Indeed, their analysis uses directly decay estimates from \cite{Herau_FP-confining} (see also  \cite{Herau-Nier,Helffer-Nier_book}) for the linearization about a stationary solution. Yet using those would prevent us from keeping track of the anisotropic nature of dissipation that helps in improving estimates of $\delta_0(\tau,R_0)$. 
Moreover, as we discuss more precisely below --- see comments around Propositions~\ref{p:remcolllin} and~\ref{p:remcoll} ---, in the limit $\tau\to0$, our optimization of involved hypocoercive Lyapunov functionals differs from what would follow from an optimal treatment of the linearized problem !
 
\subsection*{Main results}

Let us state our main results concerning the latter, that we split in cases including respectively either the strongly collisional regime or the the regime of evanescent collisions.

Our first result concerns the diffusive regime, namely when $\tau$ is small.
\bt[Diffusive regime; $\tau\lesssim 1$]~\\
For any $\tau_0>0$, any $R_0>1$ and any $\rho_*\in W^{1,p}(\T^2)$, with $p>2$, such that $\int_{\T^2}\rho_*=1$, there exist positive constants 
$\theta_0$ 
  and $K$ 
  such that for any $\tau\leq\tau_0$, any $\delta \geq K\,(1+R_0^{1/2})$, 
\begin{itemize}
\item for any $f_0$ such that 
$$
\|f_0\|_{L^2(M^{-1})}\leq R_0
\qquad\textrm{and}\qquad 
\int_{\T^2\times\R^2}f_0=1\,,
$$
then Equation~\eqref{VPFPf} possesses a (unique strong) solution $f$ starting from $f_0$, and it satisfies for any $t\geq0$
\[
 \left\|f(t,\cdot,\cdot) - \finf\right\|_{L^2(M^{-1})}\ \leq\ K\,\|f_0-\finf\|_{L^2(M^{-1})}\, e^{-\theta_0\,\tau\, t}
\]
where $\finf$ solves \eqref{globmax}-\eqref{PoissonEmdenIntro};
\item for any $f_0$, $g_0$ such that 
$$
\|f_0\|_{L^2(M^{-1})}\leq R_0\,,\quad
\|g_0\|_{L^2(M^{-1})}\leq R_0
\qquad\textrm{and}\qquad 
\int_{\T^2\times\R^2}f_0=\int_{\T^2\times\R^2}g_0=1\,,
$$
then corresponding solutions $f$ and $g$ satisfy for any $t\geq0$
\[
 \left\|f(t,\cdot,\cdot) - g(t,\cdot,\cdot)\right\|_{L^2(M^{-1})}\ \leq\ K\,\|f_0-g_0\|_{L^2(M^{-1})}\, e^{-\theta_0\,\tau\, t}\,.
\]
\end{itemize}
\label{t:maindiff}
\et

Note that forcing by $\rho_*$ induces inhomogeneity in our lower bound on admissible $\delta$s.
\smallskip

Our second result concerns the regime of evanescent collisions, namely when $\tau$ is large.
\bt[Evanescent collisions; $\tau\gtrsim 1$]~\\
For any $\eps>0$, any $\tau_0>0$, any $R_0>1$ and any $\rho_*\in W^{1,p}(\T^2)$, with $p>2$, such that $\int_{\T^2}\rho_*=1$, there exist positive constants $\theta_0$ and $K$ such that for any $\tau\geq\tau_0$, any $\delta\,\geq\,K\,(1+R_0^{1/2})\,\tau^{7/15+\eps}$, 
\begin{itemize}
\item for any $f_0$ such that 
$$
\|f_0\|_{L^2(M^{-1})}\leq R_0
\qquad\textrm{and}\qquad 
\int_{\T^2\times\R^2}f_0=1\,,
$$
then Equation~\eqref{VPFPf} possesses a (unique strong) solution $f$ starting from $f_0$, and it satisfies for any $t\geq0$
\[
 \left\|f(t,\cdot,\cdot) - \finf(\cdot,\cdot)\right\|_{L^2(M^{-1})}\ \leq\ K\,\|f_0-\finf\|_{L^2(M^{-1})}\, e^{-\theta_0\,\frac{t}{\tau}}
\]
where $\finf$ solves \eqref{globmax}-\eqref{PoissonEmdenIntro};
\item for any $f_0$, $g_0$ such that 
$$
\|f_0\|_{L^2(M^{-1})}\leq R_0\,,\quad
\|g_0\|_{L^2(M^{-1})}\leq R_0
\qquad\textrm{and}\qquad 
\int_{\T^2\times\R^2}f_0=\int_{\T^2\times\R^2}g_0=1\,,
$$
then corresponding solutions $f$ and $g$ satisfy for any $t\geq0$
\[
 \left\|f(t,\cdot,\cdot) - g(t,\cdot,\cdot)\right\|_{L^2(M^{-1})}\ \leq\ K\,\|f_0-g_0\|_{L^2(M^{-1})}\, e^{-\theta_0\,\frac{t}{\tau}}\,.
\]
\end{itemize}
\label{t:maincoll}
\et

Some comments are in order.  
Previous theorems prove that the combination of the transport term that mixes space and velocity at typical time scale of size $1$ with the Fokker-Planck part that regularizes and dissipates in the velocity variable at time scale $\tau$ does lead to both decay and regularity in all variables, regularization being somewhat implicit in our statements, but clearly apparent in our proofs. This type of structure is actually the prototype of systems leading to hypocoercive decay to equilibrium \cite{Herau-Nier,Mouhot-Neumann,Villani_hypo,Duan_hypo-linear,Dolbeault-Mouhot-Schmeiser,dMRV} and hypoelliptic regularization \cite{hormander_1967_hypoelliptic,Kohn_hypo,Herau-Nier,Helffer-Nier_book,dMRV}. Regularization allows us to obtain exponential convergence starting from initial data in $L^2(M^{-1})$ and to prove Lipschitz dependence on the initial data from the norm topology of $L^2(M^{-1})$ to the norm topology of $L^\infty(\R_+;L^2(M^{-1}))$. The latter is somewhat in contrast with the analysis in \cite{Ha-Noh-stability-VPFP-frictionless}.
The presence of a constant $K$ (larger than $1$) in our exponential decay estimates reflects both the non purely dissipative nature of the nonlinear system under consideration and, since we start from $L^2$ initial data, our use of regularizing effects. Indeed, to state it in technical words, the dissipations of functionals involved in our proofs need some time to control enough regularity to prevent the nonlinearity from inducing some norm growth.

\begin{wrapfigure}{L}{.4\textwidth}
\centering
\includegraphics[width = .4\textwidth]{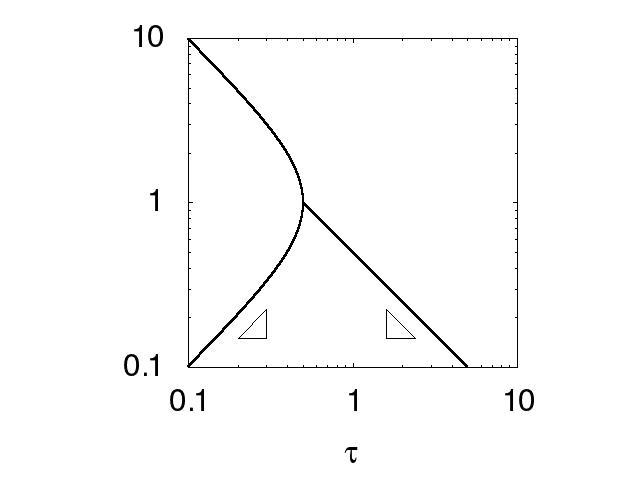}\\
\emph{Real part of eigenvalues of the toy model \emph{vs}. $\tau$, in logarithmic scale.}
\end{wrapfigure}
The dependence of decay rates in the collisional parameter $\tau$ also stems from the multi-scale anisotropic behavior of the system. As a simple but enlightening toy model consider the system of two ordinary differential equations  
$$
\begin{array}{rcl}
\hspace{.4\textwidth}X'(t)&=&V(t)\\
\hspace{.4\textwidth}V'(t)&=&-X(t) - \tau^{-1}V(t)
\end{array}
$$
that mixes $X$ and $V$ at scale $1$ and dissipates explicitly $V$ at scale $\tau$, mimicking respectively transport and collisions. For the toy model one readily checks that the rate of exponential decay to zero behaves as $1/(2\tau)$ in the limit $\tau\to\infty$ and as $\tau$ in the limit $\tau\to0^+$. This is consistent with our results for \eqref{VPFPf}. Also note that our decay rates are directly related to the spectral gap of self-adjoint operators
$
-\tau\,\Delta_x\,-\,\tau^{-1}\,\textrm{div}_{v} \left( v\cdot+\nabla_{v}\cdot\right)
$
on $L^2(M^{-1})$, that exhibit the same asymptotic behaviors.

\medskip

 To comment on constraints on $\delta$, let us start with a deliberately oversimplified analogy. Retaining from the foregoing discussion only decay rates and assuming that regularity is not an issue leads to the consideration of scalar differential equations $y'(t)\,=\, \frac{1+y(t)}{\delta^2}y(t)-\tau y(t)$ when $\tau\lesssim1$ and $y'(t)\,=\, \frac{1+y(t)}{\delta^2}y(t)-\frac{1}{\tau}y(t)$ when $\tau\gtrsim1$, the forcing by $y/\delta^2$ modeling in particular the effect of the inhomogeneity of $\rho_*$. For those equations the ball of center zero and radius $R_0$ is uniformly attracted to zero respectively when $\delta>\sqrt{1+R_0}\,\tau^{-1/2}$ and when $\delta>\sqrt{1+R_0}\,\tau^{1/2}$. 
Though those equations fail to provide relevant predictions for our system, they give an idea of the kind of conclusion that can be drawn when one first derives bounds for the semi-group evolution, without tracking anisotropic dependences in parameters, and then uses them at the nonlinear level. 
In the regime where $\tau\lesssim1$ it is obvious that the foregoing scalar equation is way too pessimistic to be relevant since it overlooks the dissipation in $v$ at rate $\tau^{-1}$ that helps to bound any term that involves a derivative in $v$, including all contributions from the electric field, hence all nonlinear terms ! It is a bit more surprising that it also fails to predict an accurate threshold in the limit $\tau\to\infty$.

It may be worth stressing that the necessity to impose $\delta\to\infty$ when $\tau\to\infty$ is a short-time/nonlinear constraint and that for the evolution linearized about $\finf$ one may obtain uniform estimates assuming only that $\delta$ is large uniformly with respect to $\tau\in(0,\infty)$. This follows directly from Proposition~\ref{p:remcolllin} and contrasts with Proposition~\ref{p:remcoll}.

To try a comparison with previous analyses involving an explicit discussion on the parameter $\tau$, let us extrapolate that when $\rho_*\equiv1$ our constraint could be turned into $\|f_0-\finf\|/\delta^2\leq K_\eps\,\min(1,\tau^{-(\tfrac{14}{15}+\eps)})$ for any $\eps>0$ and some $K$, where $\|\cdot\|$ is a suitable norm. As already pointed out, in the limit $\tau\to\infty$ this compares unfavorably with the recent analysis of Bedrossian \cite{Bedrossian_VPFP} that only requires $\|f_0-\finf\|\leq K_\delta\, \tau^{-1/3}$ (with a norm encompassing strong regularity however) when $\tau\lesssim1$, and whose cornerstones are Landau damping and mixing-enhanced dissipation --- a priori not available in our context --- that leads to a dissipation in $\tau^{-1/3}$, instead of $\tau^{-1}$, for inhomogeneities in the $x$ variable. 
In the limit $\tau\to0$ it would compare however favorably with the analysis by Jin and Zhu \cite{jin_2017_hypocoercivity}, that also appeared during the finalization of the present contribution and that requires $\|f_0-\finf\|/\delta^2\leq K\,\tau$ when $\tau\lesssim1$ (in a norm also requiring strong regularity of the initial data). In particular their result does not allow to take limits when $\tau\to0$ for a fixed nontrivial pair $(f_0,\delta)$. 
However, for precision's sake let us mention that the primarily focus of \cite{jin_2017_hypocoercivity} is on quantifying propagation of uncertainty in the initial data. 

\medskip

Now we turn to asymptotic regimes. In the diffusive regime, constraints on $\delta$ are uniform with respect to $\tau$. Therefore for any sufficiently large fixed $\delta$ one may examine model reductions in the limit $\tau\to0$. Those depend on the typical time scale chosen to observe the solution. Therefore we introduce an observation time $\tref\equiv \tref(\tau)$ and scale $f$ accordingly to obtain $\fref$ through $\fref(t,x,v)=f(\tref\,t,x,v)$. The following diagram may help the reader to visualize remarkable time scales appearing in our last main result.

\medskip
    \begin{tikzpicture}
        \def \lref{\linewidth}
     \coordinate (P1) at (-.4\lref,0);
     \coordinate (P12) at (-.3\lref,0);
     \coordinate (P2) at (-.2\lref,0);
     \coordinate (P23) at (-.05\lref,0);
     \coordinate (P3) at (.1\lref,0);
     \coordinate (P34) at (.25\lref,0);
     \coordinate (P4) at (.4\lref,0);
     \draw[->,thick] (P1)--(P4) node[right]{$t_\text{ref}(\tau)$};
     \draw (P12) node[above]{$\mathbf{(v)}$};
     \draw (P2) node{\tiny|} node[below=.3em]{$\tau$} node[above]{$\mathbf{(iv)}$};
     \draw (P23) node[above]{$\mathbf{(iii)}$};
     \draw (P3) node{\tiny|} node[below=.3em]{$1/\tau$} node[above]{$\mathbf{(ii)}$};
     \draw (P34) node[above]{$\mathbf{(i)}$};
     \draw [thick,decoration={ brace,mirror,raise=0.05\lref },decorate] (-.4\lref,0) -- (.08\lref,0);
     \draw [thick,decoration={ brace,mirror,raise=0.05\lref },decorate] (.12\lref,0) -- (.39\lref,0);
     \draw [thick,decoration={ brace,raise=0.05\lref },decorate] (-.4\lref,0) -- (-.22\lref,0);
     \draw [thick,decoration={ brace,raise=0.05\lref },decorate] (-.18\lref,0) -- (.39\lref,0);
     \draw (-.3\lref,3.5em) node[above]{Frozen};
     \draw (-.3\lref,2.5em) node[above]{initial data};
     \draw (.1\lref,2.5em) node[above]{Asymptotically thermalized regimes};
     \draw (-.15\lref,-2.5em) node[below]{Asymptotically free-field};
     \draw (-.15\lref,-3.5em) node[below]{linear regimes};
     \draw (.25\lref,-2.5em) node[below]{Asymptotically};
     \draw (.25\lref,-3.5em) node[below]{steady regime};
     \draw[->,thick] (.1\lref,-2em)--(.1\lref,-6em) node[below]{Nonlinear low-field};
     \draw[->,thick] (-.2\lref,2em)--(-.2\lref,6em) node[above]{Fokker-Planck};
     \draw (-.2\lref,7em) node[above]{Homogeneous};
     \draw (.1\lref,-7em) node[below]{regime};
     \draw (P1) node{\tiny|} node[below=.5em]{$0$};
    \end{tikzpicture} \\
    
    \vspace{2em}
    
\bt[Diffusive limits; $\tau\to0$]$ $\\
\label{t:asymptotics}
Let us denote $\fref(t,x,v)=f(\tref(\tau)\,t,x,v)$ the rescaled solution provided by Theorem~\ref{t:maindiff}.\\
\begin{itemize}
\item[\bf (i).] With $K$ depending only on $R_0$, $\rho_*$ and $\tau_0$, there holds for any $t\geq0$ and $\tau\leq\tau_0$
\[
 \left\|\fref(t,\cdot,\cdot) - \finf(\cdot,\cdot)\right\|_{L^2(M^{-1})}\ \leq\ K\,\|f_0-\finf\|_{L^2(M^{-1})}\, e^{-\theta_0\,\tref\tau\, t}
\]
so that, in $L^2(M^{-1})$, uniformly with respect to times $t$ taken in compacts of $(0,\infty]$,
$$
\fref(t,\cdot,\cdot)
\stackrel{\tau\to0}{\longrightarrow}\finf
$$
provided $\tref(\tau)\,\tau\stackrel{\tau\to0}{\to}\infty$.

\smallskip

\item[\bf (ii).] When $\tref(\tau)=\tau^{-1}$, with $K$ depending only on $R_0$, any $\rho_*\in W^{2,p}(\T^2)$, with $p>2$, and $\tau_0$, there holds for any $\tau\leq\tau_0$
\[
 \left\|\fref-\rhoref\,M\right\|_{L^2(\R_+;L^2(M^{-1}))}\ \leq\ K\,\tau\,\|f_0-\finf\|_{L^2(M^{-1})}\stackrel{\tau\to0}{\longrightarrow}0
\]
with $\rhoref=\int_{\R^2}\fref(\cdot,\cdot,v)\emph{\dD} v$, and, in $L^\infty(\R_+;L^2(\T^2))$,
$$
\rhoref(t,\cdot,\cdot)
\stackrel{\tau\to0}{\longrightarrow}\rhoas
$$
where $\rhoas\in L^\infty(\R_+;L^2(\T^2))\cap  L^2(\R_+;H^1(\T^2))$ is the unique strong solution to the drift-diffusion equation 
\be
\left\{
\begin{array}{l}
\d_t \rhoas + \textrm{div}_x(\Eas\ \rhoas-\nabla_x\rhoas) \,=\, 0\,,\\[.75em]
\Eas\,=\,\delta^{-2}\nabla_x \Delta_x^{-1}(\rhoas-\rho_*)\,,
\end{array}
\right.
\label{driftdiff}
\ee
starting from $\rho_0=\int_{\R^2}f_0(\cdot,v)\dD v\in L^2(\T^2)$.

\smallskip

\item[\bf (iii).] In $L^2(M^{-1})$ uniformly with respect to times $t$ taken in compacts of $(0,\infty)$,
$$
\fref(t,\cdot,\cdot)
\stackrel{\tau\to0}{\longrightarrow}\rho_0\,M
$$
with $\rho_0=\int_{\R^2}f_0(\cdot,v)\emph{\dD} v$ provided $\tref(\tau)\,\tau\stackrel{\tau\to0}{\to}0$ and $\tref(\tau)\tau^{-1}\stackrel{\tau\to0}{\to}\infty$\,.

\smallskip

\item[\bf (iv).] When $\tref(\tau)=\tau$, in $L^2(M^{-1})$, uniformly with respect to times $t$ taken in compacts of $[0,\infty)$,
$$
\fref(t,\cdot,\cdot)
\stackrel{\tau\to0}{\longrightarrow}\fFP
$$
where $\fFP$ is the unique strong solution to the homogeneous Fokker-Planck equation 
\be
\d_t \fFP\ =\ \textrm{div}_v(v\,\fFP+\nabla_v \fFP)\,,
\label{hFP}
\ee
starting from $f_0$.\\
\item[\bf (v).] In $L^2(M^{-1})$, uniformly with respect to times $t$ taken in compacts of $[0,\infty)$,
$$
\fref(t,\cdot,\cdot)
\stackrel{\tau\to0}{\longrightarrow}f_0
$$
provided $\tref(\tau)\tau^{-1}{\to}0$ when $\tau\to0$\,.
\end{itemize}
\et
Note that our Lipschitz dependence on initial data also allows us to replace the fixed initial data of the previous theorem with a $\tau$-dependent family converging to $f_0$ in $L^2(M^{-1})$ as $\tau\to0$.

In this theorem we prove strong convergence in each particular regime. In the asymptotically steady regime \textbf{(i)} we prove exponential convergence to the global Maxwellian uniformly with respect to $f_0$. Here uniformity in time necessarily excludes a neighborhood of initial time since the asymptotic limit loses any trace of the initial data. Likewise, in regimes \textbf{(ii)} and \textbf{(iii)} part of the initial data is asymptotically lost so that uniformity near initial time may hold for the asymptotics of $\rhoref$ but not for $\fref$. There is a threshold $\tref(\tau)=\tau^{-1}$ between asymptotic regimes where nonlinearity due to field effects play a role (\textbf{(i)} and \textbf{(ii)}) and the asymptotically field-free linear regimes  (\textbf{(iii)},  \textbf{(iv)} and \textbf{(v)}). For the latter regimes, moreover, the macroscopic density $\rho$ is asymptotically stuck to its initial data, which prevents uniformity in time to hold up to infinite time. Actually those three regimes may be understood via the fact that solutions are asymptotically close to the family of solutions $\tf$ to homogeneous Fokker-Plank equation
$$
\frac{\tau}{\tref(\tau)}\d_t \tf\ =\ \textrm{div}_v(v\,\tf+\nabla_v \tf)\,,
$$
starting from $f_0$. As for the foregoing linear case, our proof does provide convergence rates for those three asymptotically linear regimes when one assumes more regularity on initial data.

At the threshold  \textbf{(ii)} appears the most interesting limit, known as the low-field, parabolic or 
diffusion limit\footnote{In contrast we use the term for any limit corresponding to sending the diffusive parameter to infinity.}. 
Here the asymptotic dynamics for $\rhoref$ is non-linear and capable of connecting initial data to large-time equilibrium, resulting in estimates uniform with respect to time in $[0,\infty]$. Note moreover that System \eqref{driftdiff} inherits the same properties of exponential convergence and uniform stability with respect to initial data as \eqref{VPFPf}. As for the large-time limit, our main goal here is not to prove that the limit holds but to provide strong convergence in the same $L^2$ space where the initial data is taken and uniform in time. Indeed the present limit has been extensively investigated over the years, with first results obtained by Poupaud and Soler \cite{poupaud_2000_parabolic}, then improved by Goudon   \cite{goudon_2005_hydrodynamic} and extended to higher dimensions \cite{el_2010_diffusion} and multiple-species dynamics \cite{wu_2014_diffusion, herda_2016_massless}. In particular, for various variants of \eqref{VPFPf}, the limit is known to hold for initial data with finite mass, finite energy and finite entropy (plus one moment in velocity when the system is set on $\R^2$). Yet the convergence of $\rho$ proved there is local in time and only weak in the spatial variable, that is, it is proved in $L^\infty_{loc}((0,\infty);L^1-weak)$. In particular those convergence results cannot be used to transfer the large-time behavior of the limiting \eqref{driftdiff} to the original \eqref{VPFPf}. In contrast, assuming here that $\delta$ is sufficiently large, uniformly with respect to $\tau$, we prove that the convergence of $\rhoref$ holds in $L^\infty(\R_+;L^2(\T^2))$ assuming that $f_0\in L^2(M^{-1})$. Moreover, as for asymptotically linear regimes, assuming more regularity on the initial data, our proof also provides uniform convergence rates.

\subsection*{Perspectives}

To conclude this introduction, we draw now a few perspectives. 

\subsubsection*{Optimality} Probably the most interesting question left open by our analysis is the question of optimality of the constraint on $\delta$. Unfortunately our analysis provides almost no hint concerning possible scenarios for transient growth without bound on bounded sets of initial data.

\subsubsection*{Higher dimensions} Our strategy, consisting mainly in optimizing in parameters classical hypocoercive/hypoelliptic estimates, is robust enough to be adapted to most situations, with different outcome however, the optimization requiring a case-by-case analysis. It is important to note yet, that, in dimension $3$, our analysis would not allow us to start from $L^2$ initial data. This is due to insufficiently strong regularization mechanisms. To give a few insights, hold in mind that as encoded in our functionals, regularization of one space derivative costs initial blow-up $t^{-3/2}$ and of one velocity derivative costs $t^{-1/2}$ blow-up, that $E$ is one derivative smoother than $h$ in spatial variables, and that $H^s(\T^d)$ is embedded in $L^\infty(\T^d)$ provided $s>d/2$. By using these elements an analysis that would treat $E\cdot A^*h$ as a perturbative term, starting from $L^2$ initial data, as we do, would require $t\mapsto t^{(3/2)\times((d/2)+\eta-1)+1/2}$ to be locally integrable (for some $\eta>0$) which is possible only if $d=2$ and $0<\eta<1/3$. The same heuristic suggests that in dimension $3$ one needs to start with $a$ spatial derivatives and $b$ velocity derivatives with $(a,b)$ so that $3a+b>1/2$. See \cite{herau_2016_global} for some related discussions and a $3$-dimensional analysis starting with more than $1/2$-derivatives in all variables.

\subsubsection*{Whole space} We anticipate that an extension of our analysis to the whole space $\R^2$ in the presence of a confining potential, as in \cite{herau_2016_global}, would not require significant changes. More interesting would be an extension to $\R^2$ for the same system, as in \cite{hwang_2013_vlasov}. It may seem in contradiction with the fact that at several places we use a spectral gap argument for $-\Delta_x$ on $\T^2$ in the form of Poincar\'e inequalities. To bypass this difficulty one needs to be able to include lower-order terms in the dissipation to provide a direct control on the macroscopic density $n$ (without using $\nabla_x n$). At least when $\rho_*\equiv1$, the coupling by the Poisson equation allows us to carry out this plan. Indeed, a direct computation shows that
$$
\int E\cdot\nabla_x n\,=\,-\frac{1}{\delta^2}\|\nabla_x (-\Delta_x)^{-\tfrac12}n\|_{L^2}^2\,=\,-\frac{1}{\delta^2}\|n\|_{L^2}^2 
$$
and, in the case where $\rho_*\equiv1$, the foregoing term appears in the time derivative of
$$
\int j\cdot\nabla_x n
$$
so that we only need to add those two terms to our functionals and correctly tune parameters. Involved extra computations are expected to be similar to those performed in Section~\ref{s:asymptotics}. Interestingly enough, the approach shares some similarities with the method developed by Dolbeault, Mouhot and Schmeiser \cite{Dolbeault-Mouhot-Schmeiser} to obtain hypocoercive estimates in $L^2$ for equations that do not exhibit hypocoercive regularization. Indeed for the kinetic Fokker-Planck equation $\d_tf+v\cdot\nabla_xf=\textrm{div}_v(v\,f+\nabla_v f)$ their method consists in adding 
$$
\int j\cdot (\I-a\Delta_x)^{-1}\nabla_x n\,,
$$
(where $a$ is some explicit universal constant) to the usual $L^2$ energy functional. 

\subsubsection*{Other extensions} Many other extensions would require a similar case study : other collisional operators, multiple-species dynamics, coupling with magnetic fields \emph{etc.}

\subsubsection*{High-field, hyperbolic limit}  To complete the picture, let us discuss another famous asymptotic limit of \eqref{VPFPf}, the so-called high-field or hyperbolic limit. It consists in taking the limit $\tau\to0$ with $\delta^2=\tau^{-1}$ on a time scale $\tref$ independent of $\tau$. In this regime nonlinear terms dominate diffusive effects and we are asymptotically lead to an hyperbolic equation on the macroscopic density. The limit is now rigorously established \cite{nieto_2001_high, goudon_2005_multidimensional}. 
But the regime is somewhat orthogonal to our focus since it brings us to an equation that could lead to shock formation in finite time. It seems rather unclear if our approach could improve anything to our understanding of this limit. Yet it would be interesting to gain some insight on what happens asymptotically for quasineutral parameters of intermediate size, between the low-field and high-field regimes, or, even for the high-field regime, to elucidate what happens on other time scales. Any progress is likely to require a significantly more nonlinear version of our arguments, for instance involving the time-dependent measure of density $\eD^{-\phi(t,\cdot)}\,M$ instead of the stationary $\finf$.

\subsection*{Outline}
The rest of the paper is devoted to proofs of our three main theorems. 
In the next section, we introduce notational conventions and expound our strategy. 
In Section~\ref{p:PoissonBoltzmann}, we investigate the well-posedness of Equation~\eqref{PoissonEmden} and gather estimates on its solution $\phiinf$, which provides estimates on the steady state $\finf$. Then, in Section~\ref{s:Poisson}, we glean estimates on $E$ in terms of $h$ when it is obtained from the Poisson equation of \eqref{VPFPabstract}. In Section~\ref{s:frozen}, we gather some preliminary pieces of information on solutions to the system obtained by freezing nonlinear terms. In Section~\ref{s:linear}, to support our choice of exponents $\beta$ in the functional $\cF$, we close estimates for $\d_th+L_\tau h=0$, a $\delta=\infty$ version of \eqref{VPFPabstract}. Theorems~\ref{t:maindiff} and \ref{t:maincoll} are then proved in Section~\ref{s:strong-collisions} and \ref{s:evanescent}. Finally, the last section is devoted to the proof of Theorem~\ref{t:asymptotics}.

\section{Strategy}\label{s:strategy}

We now provide some more details on our strategy. To begin with, though we do consider initial data not necessarily close to $\finf$, it is convenient for this forthcoming analysis to write solutions in a seemingly perturbative form by introducing the following new unknowns
 \[
h = \frac{f-f_\infty}{f_\infty},\qquad\qquad
\psi = \phi - \phiinf\,.
 \]
In terms of $(h,\psi)$, System~\eqref{VPFPf} becomes 
\be
\d_t h + v\cdot\nabla_xh - \nabla_x\phiinf\cdot\nabla_vh + \frac{1}{\tau}
\,(v-\nabla_v)\cdot\nabla_vh + \nabla_x\psi\cdot v\ =\  \nabla_x\psi\cdot(\nabla_v-v)h
 \label{VPFPh}
\ee
coupled with the Poisson equation
\be
-\delta^2\Delta \psi\ =\ n\,,\label{Poisson}
\ee
where the source is given by
\[
n\  =\ \int_{\RM^2} h\ \finf \dD v\,.
\]

Once sufficient bounds have been obtained on $\phiinf$, one may rightfully replace in all our statements canonical norms of $L^2(M^{-1})$ with the equivalent norm arising from its interpretation as $\cH:=L^2(\T^2\times\RR^2, \mu)$ where $\mu$ is the measure with probability density function $\finf$. The space $\cH$ is thus endowed with its canonical scalar product
\[
 \lla\cdot,\cdot\rra:\,(f,g)\longmapsto \int_{{\T^2\times\RR^2}} f\,g \,\dD\mu.
\]
and we denote by $\|\cdot\|$ the corresponding norm. The main advantage is that now at the linearized level the transport term becomes skew-symmetric for the new structure whereas the Fokker-Planck operator remains symmetric. We stress that $h_0\finf$ is mean-free and so will remain $h(t,\cdot,\cdot)\finf$ for later times $t$. Hence we introduce the following subspace of $\cH$
\[
 \cH_0\ =\ \left\{h\in\cH\quad\text{such that}\quad
\lla\funcun;h\rra\,=\,\int_{\T^2\times\RR^2}\,h\,\dD\mu\, =\, 0\right\}
\]
where $\funcun$ denotes the constant function with value $1$.

As is customary in the field, in particular following the memoir of Villani \cite{Villani_hypo}, we shall write estimates proving hypocoercivity using an abstract formulation of the equations. In order to do so we introduce the following unbounded operators on $\cH$
\[
\begin{aligned}
&A&=&&&\nabla_v\,,\\
&B&=&&&v\cdot\nabla_x -(\nabla_x\phiinf)\cdot\nabla_v\,.
\end{aligned}
\]
Let us mention that in order to clarify computations using vectors or higher-order tensors we sometimes use Einstein summation convention on repeated indices.
In this way, we also introduce $L_\tau$ defined as
\[
 L_\tau\ =\ \frac{1}{\tau}A_i^*A_i + B
\]
where we use the superscript ${ }^*$ to denote coordinate-wise formal adjoint in $\cH$.
On this example our convention explicitly reads
\[
A^*\,=\,v-\nabla_v\,,\qquad
B^*\,=\,-B\,.
\]
The perturbative form \eqref{VPFPh}-\eqref{Poisson} is then equivalently written
\be
\left\{
\begin{array}{l}
\ds
\d_th\,+\,L_\tau h \,-\, E\cdot A^*(\funcun)\ =\ E\cdot A^*h\,,\\
\,\\
\ds E=\delta^{-2}\nabla_x\Delta_x^{-1}n\,,\qquad
n\  =\ \int_{\RM^2} h\ \finf \dD v\,.
 \end{array}
\right.
 \label{VPFPabstract}
\ee
In this abstract form preservation of mass follows from $A(\funcun)=0$, $B^*(\funcun)=0$, and the non-linear part of the system lies on the right-hand side of the first equation.

Commutators play a crucial role in the analysis so that we define and evaluate
\[
\begin{array}{rcccl}
C& = &[A,B] &=& \nabla_x\,,\\
&&[B,C] &=& \textrm{Hess}(\phiinf)\nabla_v\,,\\
&&[A_i,A^*_j] &=&\delta_{ij}\,,
\end{array}
\]
where $\textrm{Hess}(\phiinf)$ is the Hessian matrix of $\phiinf$ and $\delta_{ij}$ is the Kronecker symbol. At the linearized level, good dissipative terms arise from
$$
\lla h;L_\tau h\rra=\frac1\tau\|Ah\|^2
\qquad\textrm{and}\qquad
\lla ABh;Ch\rra +\lla Ah;CBh\rra
=\|Ch\|^2-\lla Ah;\textrm{Hess}(\phiinf)Ah\rra
$$
that are involved in computations of time derivatives of respectively $\|h\|^2$ and $\lla Ah,Ch\rra$, when $h$ solves~\eqref{VPFPabstract}. Incidentally, we point out that, for $K,L$ two vector-valued operators, we shall repeatedly use $KL$ to denote the matrix-valued operator with coefficients $K_iL_j$. For instance the operator yielding the Hessian in the velocity variable is denoted $AA$ or $A^2$.

\medskip

We will prove all parts of Theorems~\ref{t:maindiff} and~\ref{t:maincoll} --- existence, uniqueness, stability with respect to initial data, regularization and exponential convergence --- at once by interpreting \eqref{VPFPabstract} as the research of a fixed point for a strict contraction on a functional space that encodes regularization and decay and that quantifies precisely dependences on $\tau$. This function space is designed from functionals $\cE_{\gamma,\,\beta,\,\tau,\,\delta}$ and $\cF_{\gamma,\,\beta,\,\tau,\,\delta}^{\theta}$ built as follows. First we consider the following weighted Sobolev norm
\[
\begin{aligned}
&\nrm h \nrm_{\gamma,\,\beta,\,\tau,\,t}^2&=&&&\|h\|^2 
 + \gamma_1\tau^{\beta_1}\min\left(1,\tfrac{t}{\tau}\right)\|Ah\|^2 \\
 &&&+&& \gamma_2\tau^{\beta_2}\min\left(1,\tfrac{t}{\tau}\right)^{3}\|Ch\|^2 
 + 2\gamma_3\tau^{\beta_3}\min\left(1,\tfrac{t}{\tau}\right)^{2}\lla Ah, Ch\rra,
 \end{aligned}
\]
and a corresponding dissipation
\[
\begin{aligned}
&D_{\gamma,\,\beta,\,\tau,\,t}(h)&=&&&
\tau^{-1}\|Ah\|^2 
+ \gamma_1\tau^{\beta_1-1}\min\left(1,\tfrac{t}{\tau}\right)\|A^*\cdot Ah\|^2\\ 
&&&+&& \gamma_2\tau^{\beta_2-1}\min\left(1,\tfrac{t}{\tau}\right)^{3}\|ACh\|^2 
+ \gamma_3\tau^{\beta_3}\min\left(1,\tfrac{t}{\tau}\right)^{2}\|Ch\|^2.
\end{aligned}
\]
The presence of cross terms in $\nrm h \nrm_{\gamma,\,\beta,\,\tau,\,t}$ is related to the above mentioned commutator computation. Let us also mention that we have chosen to use weights with pure powers of $\tau$ instead of, for instance, some minima of two powers, principally to facilitate reading. However it forces us to split the discussion between regimes $\tau\gtrsim1$ and $\tau\lesssim1$. 
We mention that after the completion of our work, Cl\'ement Mouhot pointed to us that a similar strategy has been used recently by Briant \cite{briant_2015_boltzmann} to quantify hydronamic limits starting from the Boltzmann equation. 

Note that an estimation of $\nrm h(t,\cdot,\cdot) \nrm_{\gamma,\,\beta,\,\tau,\,t}$ in terms of $\nrm h_0 \nrm_{\gamma,\,\beta,\,\tau,\,0}=\|h_0\|$ when $h$ solves~\eqref{VPFPabstract} would encode hypoelliptic regularization. Powers of the time variable should be appreciated with this in mind as they are associated with classical gain of regularity afforded by the kinetic Fokker-Planck operator; see for instance \cite[Appendix~A.21]{Villani_hypo}. Furthermore, if one proves that solutions to~\eqref{VPFPabstract} satisfy for any $t_2\geq t_1\geq \tau$
$$
\nrm h(t_2)\nrm_{\gamma,\,\beta,\,\tau,\,t_2}^2
\,+\,\theta\,\int_{t_1}^{t_2}\,D_{\gamma,\,\beta,\,\tau,\,s}(h(s))^2\dD s\,\leq\,\nrm h(t_1)\nrm_{\gamma,\,\beta,\,\tau,\,t_1}^2
$$
for some $\theta>0$, then by using Poincar\'e inequality one deduces exponential decay with rates scaling as $\min(\tau^{-1},\tau^{\beta_3},\tau^{\beta_3-\beta_2})$. Hence closing such form of estimates will prove at the same time hypoelliptic regularization and hypocoercive decay.

In Section~\ref{s:linear} we first show what choices of parameters $\gamma$ and $\beta$ are available for the reduced toy system $\d_th+L_\tau h=0$ that may be thought as a $\delta=\infty$ version of System~\eqref{VPFPabstract}. We may then analyze for the original problem what is, among available parameters for the toy system, the optimal choice to relax as much as possible the constraint on $\delta$ and still obtain the same entropy/dissipation relations. Since the way in which we prove corresponding nonlinear a priori estimates lends itself to a strict contraction formulation this will lead to Theorems~\ref{t:maincoll} and~\ref{t:maindiff}. Actually in our study of \eqref{VPFPabstract} we rather use
\[
 \cE_{\gamma,\,\beta,\,\tau,\,\delta,\,t}(h) = \nrm h \nrm_{\gamma,\,\beta,\,\tau,\,t}^2 + \delta^2(1 + \gamma_1\tau^{\beta_1}\min\left(1,\tfrac{t}{\tau}\right))\|E\|_{L^2}^2
\]
to offer a better account of electric-field contributions. Our goal is then essentially to build solutions such that 
$$
\cF_{\gamma,\,\beta,\,\tau,\,\delta}^{\theta}(h)\,\leq\,K\,\|h_0\|^2
$$
for some constant $K$, where 
\be
\cF_{\gamma,\,\beta,\,\tau,\,\delta}^{\theta}(h)\ =\ \cE_{\gamma,\,\beta,\,\tau,\,\delta}(h) + \theta\,D_{\gamma,\,\beta,\,\tau}(h)
\label{func}
\ee
with 
\[
\cE_{\gamma,\,\beta,\,\tau,\,\delta}(h) \ =\ \sup_{t\geq0}\, \cE_{\gamma,\,\beta,\,\tau,\,\delta,\,t}(h(t,\cdot,\cdot)),
\]
and
\[
D_{\gamma,\,\beta,\,\tau}(h)\ =\ \int_0^\infty D_{\gamma,\,\beta,\,\tau,\,t}(h(t,\cdot,\cdot))\,\dD t
\]
for a suitable choice of $\theta>0$, $\beta\in\R^3$ and  $\gamma\in(0,+\infty)^3$ under the weakest possible constraint on $\delta$ and uniformly with respect to $(\tau,h_0)$ taken in relevant spaces. It turns out that our choices are $\beta=(0,2,1)$ when $\tau\lesssim1$ and $\beta=(-8/15,2/5,-1/15)$ when $\tau\gtrsim1$, and we amply comment on motivations of these choices along the proof. We stress however here that the latter choice differ from the choice $\beta=(-1,-1,-1)$ optimal for the linearized dynamics when $\tau\gtrsim1$.

Once Theorem~\ref{t:maindiff} and~\ref{t:maincoll} have been proved, corresponding estimates or higher-order versions of those may be used to bound error terms in diffusive asymptotics, leading to Theorem~\ref{t:asymptotics}.

\section{The Poisson-Boltzmann equation}\label{p:PoissonBoltzmann}

In this section we provide well-posedness and regularity results for 
 \be
  -\delta^2\Delta_{x}\phiinf = e^{-\phiinf} - {\rho_*}\,.
  \label{PoissonEmden}
 \ee
Consistently with our global analysis we insist on uniformity of estimates with respect to $\delta$ when $\delta$ is bounded away from zero. This turns out to be crucial so as to control all our norms and relative inequalities since they depend on $\delta$ through $\phiinf$.

\medskip

As a key example note the following form of the Poincar\'e inequality in $\cH_0$.
\begin{proposition}\label{prop:est-poincare}
There exists a positive constant $K$ such that for any $h\in\cH_0$ and any $\delta>0$, one has
$$
\|h\|^2
\,\leq\,K\,\|e^{\phiinf}\|_{L^\infty(\T^2)}\,\|e^{-\phiinf}\|_{L^\infty(\T^2)}\left(\|Ah\|^2+\|Ch\|^2\right)\,.
$$ 
\end{proposition}
\noindent The foregoing inequality is actually a straightforward consequence of the tensorization of the classical Poincaré inequality on the torus with the Gaussian Poincaré inequality. The reader is referred to \cite[Chapter~4]{bakry_2013_analysis} for a detailed argument.

\medskip

In the present section our arguments are relatively classical and strongly echo those in \cite{bouchut_1991_global,Dolbeault_stationary-states,CDMS} and even more those in \cite[Section~3]{herda_2016_anisotropic}. In particular the existence part follows by identifying~\eqref{PoissonEmden} with an Euler-Lagrange equation. In order to do so we set $n_h=\rho_*-1$ so that $n_h$ is mean-free and introduce the functional 
$$
J(\phi)\ =\ \frac{\delta^2}{2}\int_{\TT^2}|\nabla_x\phi|^2\,+\,\int_{\TT^2}\phi\,n_h\,+\,
\ln\left(\int_{\TT^2}e^{-\phi}\right)
$$
on 
$$H_0=\left\{\ \phi\in H^1(\TT^2)\ \middle|\ \int_{\TT^2}\phi=0\ \right\}\,.$$
The functional $J$ is coercive, bounded from below and strictly convex provided that $\rho_*\in H^{-1}(\TT^2)$. The main observation leading to strict convexity is that from the Holder inequality stems for any $\theta\in[0,1]$ and $\phi_1, \phi_2\in H_0$,
	\[
	\int_{\T^2} e^{-\theta\phi_1 -(1-\theta)\phi_2} \dD x\leq \left(\int_{\T^2} e^{-\phi_1} \dD x\right)^\theta\left(\int_{\T^2} e^{-\phi_2} \dD x\right)^{1-\theta}.
	\]
In turn, since by Jensen's inequality, for any $\phi\in H_0$
$$
\ln\left(\int_{\TT^2}e^{-\phi}\right)\,\geq\,\ln\left(e^{-\int_{\TT^2}\phi}\right)=0\,,
$$
coercivity and boundedness from below are explicitly derived from
\be
J(\phi)\,\geq\, \frac{\delta^2}{2}\int_{\TT^2}|\nabla_x\phi|^2\,-\,K\,\|n_h\|_{H^{-1}(\TT^2)}\,\left(\int_{\TT^2}|\nabla_x\phi|^2\right)^{1/2}
\,+\,\ln\left(\int_{\TT^2}e^{-\phi}\right)
\label{coerc}
\ee
that holds for some constant $K$ and any $\phi\in H_0$. The Euler-Lagrange equation associated with $J$ is actually
$$
 -\delta^2\Delta_{x}\phi_\delta = 
\frac{e^{-\phi_\delta}}{\int_{\TT^2}e^{-\phi_\delta}}-\rho_*\,.
$$
Yet, solutions $\phi_\delta\in H_0$ to the foregoing equation are in one-to-one correspondence with solutions $\phiinf\in H^1(\T^2)$ to Equation~\eqref{PoissonEmden} through 
\be
\phiinf\ =\ \phi_\delta\ +\ \ln\left(\int_{\TT^2}e^{-\phi_\delta}\right)\,,
\qquad\qquad
\phi_\delta\ =\ \phiinf\ -\ \int_{\TT^2}\phiinf\,,
 \label{recast}
\ee
(since Equation~\eqref{PoissonEmden} implicitly contains $\int_{\TT^2}e^{-\phiinf}=1$).

\bpr[Existence, uniqueness and regularity]$ $\\[-0.5em]
\begin{enumerate}
\item For any $\rho_*\in H^{-1}(\T^2)$ such that $\int_{\T^2}\rho_*=1$, for any $\delta>0$, Equation~\eqref{PoissonEmden} possesses a unique weak solution $\phiinf\in H^1(\T^2)$ and this solution is such that $\int_{\TT^2}e^{-\phiinf}=1$.
\item Moreover there exists a positive constant $K$ such that for any such $\rho_*$ and any $\delta>0$, the corresponding solution $\phiinf$ satisfies
$$
\delta^2\|\nabla\phiinf\|_{L^2(\T^2)}^2
\,+\,\left|\int_{\T^2}\phiinf\right|
\,\leq\,K\|\rho_*-1\|_{H^{-1}(\T^2)}^2\,.
$$
\item If additionally, for some  $p\in[1,+\infty]$, $\rho_*\in L^p(\T^2)$ then
      \[
 \|e^{-\phiinf}\|_{L^p(\T^2)}\,\leq\,\|\rho_*\|_{L^p(\T^2)}\,.
\]
\item In particular, for any $p\in(1,+\infty)$, there exists $K_p>0$ such that for any $\rho_*\in L^p(\T^2)$ such that $\int_{\T^2}\rho_*=1$ and any $\delta>0$, the unique solution $\phiinf$ to Equation~\eqref{PoissonEmden} satisfies
$$
  \|\nabla_x^{2}\phiinf\|_{L^p(\T^2)}\,\leq\, \frac{K_p}{\delta^2}\|\rho_*\|_{L^p(\T^2)}
$$
and
\[
\|\phiinf\|_{L^\infty(\T^2)}\,\leq\,\frac{K_p}{\delta^2}\|\rho_*\|_{L^p(\T^2)}\,.
\]
\label{regPB}
\end{enumerate}
      \label{existPB}
\epr
\begin{proof}
Existence and uniqueness follow from the properties of $J$ expounded above through a direct minimization of the strictly convex functional $J$. Then the bound in $H^1$ stems from \eqref{coerc} and $J(0)=0$ by noticing that 
$$\int_{\T^2}\phiinf\,=\,\ln\left(\int_{\TT^2}e^{-\phi_\delta}\right)\geq0\,.
$$ 
 
Concerning $L^p$ estimate of $e^{-\phiinf}$, the formal argument proceeds by multiplying the equation by $-e^{-(p-1)\phiinf}$ and integrating to derive 
\[
  (p-1)\delta^2\int_{\T^2}|\nabla_x\phiinf|^2\,e^{-(p-1)\phiinf} + \int_{\T^2}e^{-p\phiinf}
\,=\,\int_{\T^2}e^{-(p-1)\phiinf}\,\rho_*
\]
that implies
$$
\|e^{-\phiinf}\|_{L^p(\T^2)}^p\leq \|e^{-\phiinf}\|_{L^p(\T^2)}^{p-1}\|\rho_*\|_{L^p(\T^2)}
$$
by H\"older's inequality, from which the bound follows by simple computations. This may be turned into a sound argument by testing instead against $-e^{-(p-1)\max(\{\phiinf,\eta\})}$ and letting $\eta\to-\infty$. 

From here the $W^{2,p}$ bound stems directly from the equation and classical elliptic  regularity properties --- in Calder\'on-Zygmund form --- for which we refer the reader to \cite{Stein_singular-integrals,Stein-Weiss_Fourier} or \cite{Grafakos_I}. The $L^\infty$ bound then follows from the bound on $\int_{\T^2}\phiinf$ and a Sobolev embedding applied to $\phiinf-\int_{\T^2}\phiinf$.
\end{proof}

The foregoing proposition provides an $L^\infty$ bound on $e^{\phiinf}$ that blows up exponentially in $\delta^{-2}$ in the quasi-neutral regime $\delta\rightarrow 0$. Though this will be sufficient for our general argument and thus we do not pursue this line of investigation here, let us mention for precision's sake that the bound on $e^{\phiinf}$ may be dramatically improved if $\rho_*$ is bounded away from zero.

\bpr[Higher regularity]
Let $s\in\N$, $s\geq2$, $q\in(1,\infty]$ and $p\in(1,\infty)$. There exists a positive constant $K=K_{p,q,s}$ and an integer $\alpha_s$ such that for any $\rho_*\in W^{s-2,p}(\T^2)\cap L^\infty(\T^2)$ such that $\int_{\T^2}\rho_*=1$ and any $\delta>0$, the unique solution $\phiinf$ to Equation~\eqref{PoissonEmden} satisfies
\be
\|\nabla_x^{s}\phiinf\|_{L^p(\T^2)}
\,\leq\,
\frac{K}{\delta^2}\left(1+\|e^{-\phiinf}\|_{L^\infty}^{\frac{s-2}{2}}\right)\,\left(1+\|\phiinf\|_{W^{2,q}}^{\alpha_s}\right)\,
\left(1+\tfrac{1}{\delta^{s-2}}\right)\,\|\rho_*\|_{W^{s-2,p}}\,.
  \label{estimPB}
\ee
\label{regPBHs}
 \epr
\begin{proof}
We proceed by induction. The induction estimate
$$
\delta^2\|\nabla^{k+2}\phiinf\|_{L^p(\T^2)}
\,\leq\,K\|\nabla^{k}\rho_*\|_{L^p(\T^2)}
\,+\,K\|e^{-\phiinf}\|_{L^\infty(\T^2)}(1+\|\phiinf\|_{W^{2,q}(\T^2)}^{k-1})
\|\phiinf\|_{W^{k,p}(\T^2)}
$$
is obtained essentially by differentiating the equation and applying suitable Sobolev inequalities in Gagliardo-Nirenberg's form. Namely, the elementary block leading to the foregoing estimates is that if $1\leq\ell\leq k$ and $\sigma\in (\N^*)^\ell$ is such that $|\sigma|=k$ then 
$$
\begin{array}{rcl}\ds
\|\prod_{j=1}^\ell \d^{\sigma_j}\phiinf\|_{L^p(\T^2)}
&\leq&\ds 
\prod_{j=1}^\ell \|\d^{\sigma_j}\phiinf\|_{L^{p_j}(\T^2)}\\
&\leq&K\,\ds
\prod_{j=1}^\ell \|\phiinf\|_{W^{2,q}(\T^2)}^{1-\frac{\sigma_j}{k}}\,\|\nabla^k\phiinf\|_{L^p(\T^2)}^{\frac{\sigma_j}{k}}=\|\phiinf\|_{W^{2,q}(\T^2)}^{\ell-1}\|\nabla^k\phiinf\|_{L^p(\T^2)}\\
&\leq&K\,(1+\|\phiinf\|_{W^{2,q}(\T^2)}^{k-1})\,\|\phiinf\|_{W^{k,p}(\T^2)}
\end{array}
$$
where $p_j=p\,k/\sigma_j$. Actually some of the derivatives needs first to be replaced with finite differences to justify formal manipulations, but we skip those classical details.
\end{proof}

From now on we shall always assume but never repeat that $\phiinf$ is obtained from $\rho_*$ through Equation~\eqref{PoissonEmden} and we shall keep the dependence on norms of $\rho_*$ implicit. Also for concision's sake we shall use without mention estimates of the present section.

\section{The Poisson equation}\label{s:Poisson}
We glean here estimates on $E$ in terms of $h$ when $E=-\nabla\psi$ and
$$
-\delta^2\Delta \psi\ =\ n\,,\qquad n\  =\ \int_{\RM^2} h(\cdot,v)\ \finf(\cdot,v) \dD v 
$$
whenever $h\in\cH_0$. They are naturally obtained from classical estimates on the Poisson equation on one side and estimates on $n$ in terms of $h$ on the other side. We recall that $\|\cdot\|$ denotes the canonical $L^2(\dD\mu)$ norm for the measure $\dD\mu\,=\,\finf\dD x\dD v$, that does depend on $\delta$ and $\rho_*$.

\begin{proposition}$ $\\[-0.5em]
\label{prop:est-density}\label{prop:est-Poisson}
\begin{enumerate}
\item For any $s\in\N$, $\rho_*\in H^{s}(\T^2)$ such that $\int_{\T^2}\rho_*=1$ and $\delta_0>0$ there exists $K_s>0$ such that for all $\delta\geq\delta_0$, for any $h$, $n=\int_{\RM^2} h(\cdot,v)\ \finf(\cdot,v) \emph{\dD} v$ satisfies
$$
\|\nabla_x^sn\|_{L^2(\T^2)}\ \leq\ K_s\,\sum_{k\leq s}\|\nabla_x^k\,h\|
$$
and if moreover $h\in\cH_0$ and $E$ is the corresponding electric field
$$
\|\nabla_x^{s+1}E\|_{L^2(\T^2)}\ \leq\ \frac{K_s}{\delta^2}\,\sum_{k\leq s}\|\nabla_x^k\,h\|\,.
$$ 
\item For any $\rho_*\in H^{1}(\T^2)$ such that $\int_{\T^2}\rho_*=1$, $p\in(1,\infty)$ and $\delta_0>0$ there exists $K_p>0$ such that for all $\delta\geq\delta_0$, for any $h$, $n=\int_{\RM^2} h(\cdot,v)\ \finf(\cdot,v) \emph{\dD} v$ satisfies
$$
\|n\|_{L^p(\T^2)}\ \leq\ K_p\,\|h\|^{\frac2p}\,\|\nabla_x h\|^{1-\frac2p}
$$ 
and if moreover $h\in\cH_0$ and $E$ is the corresponding electric field
$$
\|\nabla_x E\|_{L^p(\T^2)}\ \leq\ \frac{K_p}{\delta^2}\,\|h\|^{\frac2p}\,\|\nabla_x h\|^{1-\frac2p}\,.
$$ 
\item For any $\rho_*\in L^{2}(\T^2)$ such that $\int_{\T^2}\rho_*=1$, $p\in(1,\infty)$ and $\delta_0>0$ there exists $K_p>0$ such that for all $\delta\geq\delta_0$ and any $h\in\cH_0$, the corresponding electric field $E$ satisfies
$$
\|E\|_{L^p(\T^2)}\ \leq\ \frac{K_p}{\delta^2}\,\|h\|\,.
$$ 
\item For any $\rho_*\in L^{2}(\T^2)$ such that $\int_{\T^2}\rho_*=1$, $\eta\in(0,1]$ and $\delta_0>0$ there exists $K_\eta>0$ such that for all $\delta\geq\delta_0$ and any $h\in\cH_0$, the corresponding electric field $E$ satisfies
$$
\|E\|_{L^\infty(\T^2)}\ \leq\ \frac{K_\eta}{\delta^2}\,\|h\|^{1-\eta}\,\|\nabla_x h\|^{\eta}\,.
$$ 
\end{enumerate}
\end{proposition}

\begin{proof}
The Sobolev estimate on the macroscopic density stems from an integration of the point-wise
$$
|\nabla^s n(x)|^2\,\leq\,K\,\sum_{k\leq s}
|\nabla_x^{s-k}(e^{-\phiinf})(x)|^2\,e^{\phiinf(x)}\,
\int_{\RM^2} |\nabla_x^k h(x,v)|^2\ \finf(x,v) \dD v
$$
that follows from direct differentiation and Jensen's inequality (for the square function). The Lebesgue estimate on the macroscopic density follows from 
$$
\begin{array}{rcl}\ds
\|n\|_{L^p(\T^2)}&\leq&\ds
\|e^{-\phiinf}\|_{L^\infty(\T^2)}
\int_{\R^2}\|h(\cdot,v)\|_{L^p(\T^2)}\,M(v)\dD v\\
&\leq&\ds K_p\|e^{-\phiinf}\|_{L^\infty(\T^2)}
\int_{\R^2}\|h(\cdot,v)\|_{L^2(\T^2)}^{\frac2p}
\|\nabla_xh(\cdot,v)\|_{L^2(\T^2)}^{1-\frac2p}\,M(v)\dD v\\
&\leq&\ds K_p\|e^{-\phiinf}\|_{L^\infty(\T^2)}\,\|e^{\phiinf}\|_{L^\infty(\T^2)}\,
\|h\|^{\frac2p}\,\|\nabla_x h\|^{1-\frac2p}
\end{array}
$$
that is derived by triangle inequality, some Sobolev embeddings and the H\"older inequalities. Remaining estimates are then deduced from classical elliptic regularity. Note in particular that for any $1\leq p<2<q\leq\infty$ there exists $K_{p,q}$ such that
$$
\|E\|_{L^\infty(\T^2)}\leq 
K_{p,q} \|n\|_{L^p(\T^2)}^{\theta_{p,q}}\|n\|_{L^q(\T^2)}^{1-\theta_{p,q}}
$$
where $\theta_{p,q}\in(0,1)$ is defined by $1/2=\theta_{p,q}/p+(1-\theta_{p,q})/q$.
\end{proof}


\section{Frozen equations}\label{s:frozen}

As preliminaries to nonlinear final arguments, here we consider nonlinear terms as source terms and derive corresponding estimates. Namely we set
\be
\d_th\,+\,
L_\tau h - E\cdot v\,=\,
\tE\cdot A^*g\,,
\label{e:frozen}
\ee
where
$$
E\,=\,
\delta^{-2}\nabla_x\Delta_x^{-1}n\,,\qquad
n\,=\,\int_{\RM^2} h\ \finf \dD v\,,
$$
and $\tE$ and $g$ are given sources, $\tE$ being derived from some given $\tg\in\cH_0$ through
$$
\tE\,=\,\delta^{-2}\nabla_x\Delta_x^{-1}\tn\,,\qquad
\tn\,=\,\int_{\RM^2} \tg\ \finf \dD v\,.
$$

We first collect algebraic identities describing each elementary piece of the final energy estimate.

\bl\label{lem:defQR}
Any smooth localized $h$ solving \eqref{e:frozen} satisfies
$$
\begin{array}{rcl}
\ds
\frac 12\frac{\emph{\dD}}{\emph{\dD} t}\left(\|h\|^2 + \delta^2\|E\|^2_{L^2}\right) 
+\frac{1}{\tau}\|Ah\|^2 
&=&\ds 
\cR_0(h,g,\tg)\,,\label{est0}\\[0.5em]
\ds
\frac 12\frac{\emph{\dD}}{\emph{\dD} t}\left(\|Ah\|^2 + \delta^2\|E\|^2_{L^2}\right)
+\frac{1}{\tau}(\|Ah\|^2 + \|A^2h\|^2)
&=&\ds
Q_A(h)+\cR_A(h,g,\tg)\,,\label{estv}\\[0.5em]
\ds
\frac 12\frac{\emph{\dD}}{\emph{\dD} t}\|Ch\|^2
+\frac{1}{\tau}\|A Ch\|^2
&=&\ds
Q_C(h)+\cR_C(h,g,\tg)\,,\label{estx}\\[0.5em]
\ds
\frac{\emph{\dD}}{\emph{\dD} t}\lla Ah, Ch\rra
+\|Ch\|^2
&=&\ds
Q_{AC}(h)+\cR_{AC}(h,g,\tg)\,,\label{estvx}
\end{array}
$$
with quadratic terms given by 
\begin{eqnarray}
\ds
Q_A(h)&=&-\lla Ch, Ah\rra\label{QA}\\
Q_C(h)&=&\lla \textrm{Hess}(\phiinf)Ah,Ch\rra + \lla \nabla_x(E\cdot v), Ch\rra\label{QC}\\
Q_{AC}(h)&=&-\frac{1}{\tau}\left(\lla Ah, Ch\rra + 2\lla A^2h, A Ch\rra\right)\ +\ \tQ_{AC}(h)\label{QAC}
\end{eqnarray}
where 
$$
\tQ_{AC}(h)\,=\,\lla  \textrm{Hess}(\phiinf) Ah,Ah\rra+\lla E, Ch\rra + \lla \nabla_x(E\cdot v), Ah\rra
$$
and trilinear terms 
\begin{eqnarray}
\ds
\cR_0(h,g,\tg)&=&\ds\lla Ah,\tE g\rra\label{R0}\\
\cR_A(h,g,\tg)&=&\ds\lla Ah,\tE g\rra+\lla A^2h,\tE Ag\rra\label{RA}\\
\cR_C(h,g,\tg)&=&\ds\lla ACh,\tE Cg\rra+\lla CAh,g\nabla_x\tE\rra\label{RC}\\
\quad\cR_{AC}(h,g,\tg)&=&\ds
\lla A^2h,\tE Cg\rra+\lla A^2h,g\nabla_x\tE\rra
+\lla Ch,\tE g\rra+\lla ACh,\tE Ag\rra\,.\label{RAC}
\end{eqnarray}
\el

For latter use we also observe that if more frozen nonlinear terms of the same form were added to \eqref{e:frozen}, this would only result in adding more trilinear terms of the same form in the foregoing identities. Namely, if $\tE\cdot A^*g$ is replaced with $\tE_1\cdot A^*g_1+\tE_2\cdot A^*g_2$, then accordingly trilinear terms $\cR_\#(h,g,\tg)$ are turned into $\cR_\#(h,g_1,\tg_1)+\cR_\#(h,g_2,\tg_2)$.

\begin{proof}
We evaluate the time derivative of $\|h\|^2$, $\|Ah\|^2$, $\|Ch\|^2$ and $\lla Ah,\,Ch\rra$. Using \eqref{e:frozen}, we need to compute three kinds of terms, involving respectively $L_\tau$, $E\cdot v$ and the nonlinear product $\tE\cdot A^*g$. Trilinear remainders $\cR_0$, $\cR_A$, $\cR_C$ and $\cR_{AC}$  are exclusively obtained from the latter while dissipation terms and quadratic remainders come from other terms.

To obtain the first energy equality we use the skew-symmetry of $B$ to derive 
\[
-\lla L_\tau h,\ h\rra\,=\,-\frac{1}{\tau}\lla A^*A h,\ h\rra\,=\,-\frac1\tau \|Ah\|^2,
\]
To proceed we use the continuity equation
\be\label{e:n}
\d_t n \,+\,\textrm{div}_x(j)\,=\,0\,,\qquad j\,:=\,\int_{\R^2} v\,h\,\finf \dD v
\ee
obtained by multiplying the first equation of \eqref{e:frozen} by $\finf$ and integrating with respect to the velocity variable. By the Poisson equation, written in terms of a potential $\psi$ such that $E=-\nabla \psi$, and \eqref{e:n} we obtain
\[
 \lla E\cdot v ,\ h\rra\,=\,\int_{\TT^2} j\cdot E\,\dD x\,=\,-\int_{\TT^2} \psi\,\partial_t n\,\dD x\,=\,-\delta^2\int_{\TT^2} \nabla_x\psi\cdot\partial_t \nabla_x\psi\, \dD x
\,=\,-\frac12\delta^2\frac{\dD}{\dD t}\|E\|_{L^2}^2\,.
\]

The second equation follows from similar computations using that $C=[A,B]$ and $\I=[A,A^*]$, which implies
\[
\begin{aligned}
 &\|A^*\cdot Ah\|^2&=&&&\lla A^*_iA_ih, A^*_jA_jh \rra\\
 &&=&&&\lla A_jA^*_iA_ih, A_jh \rra\\
 &&=&&&\lla \delta_{ij}A_ih, A_jh \rra + \lla A_jA_ih, A_iA_jh \rra\\
 &&=&&&\|Ah\|^2 + \|A^2h\|^2\,.
 \end{aligned}
\] 
This leads to
\[
\begin{aligned}
 &-\lla AL_\tau h,\ Ah\rra&=&&& -\frac{1}{\tau}\|A^*Ah\|^2-\lla ABh,\ Ah\rra\\
 &&=&&& -\frac{1}{\tau}(\|Ah\|^2 + \|A^2h\|^2)-\lla Ch,\ Ah\rra,
\end{aligned}
\]
and we also observe that
\[
\lla A_i(E_j\, v_j),\ A_ih\rra\,=\,\lla E_i ,\ A_ih\rra\,=\,\lla A_i^*(E_i),\ h\rra = \lla v\cdot E ,\ h\rra,
\]
which we recognize as a term already computed. The trilinear term $\cR_A$ is
$$
\lla A_ih,\ A_i(\tE_j\, A^*_jg)\rra\,=\lla A_ih,\ A_i\,A^*_j\,\tE_jg)\rra\,=\,\lla A_ih,\ \tE_i g\rra+\lla A_i\,A_jh,\ \tE_j\, A_ig\rra\,.
$$

Concerning the third equation we first compute 
\[
 -\lla C_i\,L_\tau h,\ C_ih\rra
= -\frac{1}{\tau}\lla A_j\,C_ih,\ A_j\,C_ih\rra + \lla [B,C_i]h,\ C_ih\rra,\\
\]
and use that $[B,C] = \textrm{Hess}(\phiinf)\,A$. The remainder term $\cR_C$ is
\[
 \lla  C_ih,\ C_i\,(\tE_j\,A_j^*g)\rra\,=\,\lla C_ih,\ (\partial_{x_i}\tE_j)\,A^*_jg\rra + \lla C_ih,\  \tE_j\,C_i\,A^*_j g\rra.
\]

Finally, for the fourth equality, we only explain how a few typical terms are derived, other following by the same kind of arguments. To this purpose note that
$$
-\lla ABh,\  Ch\rra-\lla Ah,\  CBh\rra\,=\,-\|Ch\|^2+\lla Ah,\  [B, C]h \rra
$$
and
$$
-\lla A\,A^*\cdot Ah,\  Ch\rra-\lla Ah,\  CA^*\cdot Ah\rra
\,=\,-\lla Ah,\  Ch\rra - 2\lla A^2h,\  ACh\rra\,.
$$
\end{proof}

In the following lemmas, we indicate how to estimate each right-hand side term of the foregoing lemma. 
The first lemma is a trivial corollary of the Cauchy-Schwarz inequality, 
estimates on $E$ and $\phiinf$ 
and the fact that $v\mapsto v$ lies in $L^2(M(v)\, \dD v)$. 
\bl
For any $\rho_*\in W^{1,p}(\T^2)$, $p>2$ such that $\int_{\T^2}\rho_*=1$ and $\delta_0>0$ there exists $K>0$ such that for all $\delta\geq\delta_0$, for any $h$
$$
\begin{array}{rcl}
\ds
|Q_A(h)|&\leq&\ds\|Ch\|\,\|Ah\|\\[0.5em]
|Q_C(h)|&\leq&\ds\frac{K}{\delta^2}\left(\|Ah\|\,\|Ch\|+\|h\|\,\|Ch\|\right)\\[0.75em]
|Q_{AC}(h)-\tQ_{AC}(h)|&\leq&\ds\frac{1}{\tau}\left(\|Ah\|\|Ch\|+ 2\|A^2h\|\|A Ch\|\right)\,\\[0.75em]
|\tQ_{AC}(h)|&\leq&\ds\frac{K}{\delta^2}\left(\|Ah\|^2+\|h\|\,\|Ch\|+\|h\|\|Ah\|\right)\,,
\end{array}
$$
where terms on the left-hand sides are defined as in \eqref{QA}-\eqref{QC}.
\label{lem:estQ}
\el

Now we estimate trilinear terms by norms involving $h,\,g,\,\tg$ and their derivatives.
\bl
For any $\rho_*\in L^{2}(\T^2)$ such that $\int_{\T^2}\rho_*=1$, $\eta\in(0,1)^4$ and $\delta_0>0$ there exists $K=K_{\eta}>0$ such that for all $\delta\geq\delta_0$ and any $(h,g,\tg)\in(\cH_0)^3$,
\[
\begin{array}{rcl}
\ds
|\cR_0(h,g,\tg)|&\leq&\ds \frac{K}{\delta^2}\ \|Ah\|\ \min(\|g\|^{1-\eta_1}\ \|Cg\|^{\eta_1}\ \|\tg\|,\ \|g\|\ \|\tg\|^{1-\eta_2}\ \|C\tg\|^{\eta_2})\,,\\
\,\\[0.75em]\ds
|\cR_A(h,g,\tg)|&\leq&\ds\frac{K}{\delta^2}\ \|Ah\|\ \|g\|^{1-\eta_1}\ \|Cg\|^{\eta_1}\ \|\tg\|\\[0.75em]\ds
&+&\ds
\frac{K}{\delta^2}\ \|A^2h\|\ \|\tg\|^{1-\eta_2}\ \|C\tg\|^{\eta_2}\ \|Ag\|\,,\\
\,\\[0.75em]\ds
|\cR_C(h,g,\tg)|&\leq&\ds\frac{K}{\delta^2}\ \|ACh\|\ \|Cg\|\ \|\tg\|^{1-\eta_1}\ \|C\tg\|^{\eta_1}\\[0.75em]
&+&\ds\frac{K}{\delta^2}\ \|ACh\|\ \|\tg\|^{1-\eta_2}\ \|C\tg\|^{\eta_2}\ \|g\|^{\eta_2}\ \|Cg\|^{1-\eta_2}\,,\\
\,\\[0.75em]\ds
\end{array}
\]
\[
\begin{array}{rcl}
|\cR_{AC}(h,g,\tg)|&\leq&\ds
\frac{K}{\delta^2}\ \|A^2h\|\ \|\tg\|^{1-\eta_1}\|C\tg\|^{\eta_1}\ \|Cg\|\\[0.75em]
&+&\ds
\frac{K}{\delta^2}\ \|A^2h\|\ \|\tg\|^{1-\eta_2}\ \|C\tg\|^{\eta_2}\ \|g\|^{\eta_2}\ \|Cg\|^{1-\eta_2}\\[0.75em]
&+&\ds
\frac{K}{\delta^2}\ \|Ch\|\ \|g\|^{1-\eta_3}\ \|Cg\|^{\eta_3}\ \|\tg\|\\[0.75em]
&+&\ds
\frac{K}{\delta^2}\ \|ACh\|\ \|\tg\|^{1-\eta_4}\|C\tg\|^{\eta_4}\ \|Ag\|\,.\\[0.5em]
\end{array}
\]
\label{lem:estR}
where terms on the left-hand sides are defined as in \eqref{R0}-\eqref{RAC}.
\el

\begin{proof}
Half of the estimate on $\cR_0$ stems directly from bounds on $\|\tE\|_{L^\infty(\TT^2)}$. The other half follows, setting $p=2/\eta\in(2,\infty)$ and defining $q\in(2,\infty)$ by $1/2=1/p+1/q$, from
\[
\begin{aligned}
 &\left|\lla Ah,\tE g\rra\right|&\leq&&&K\,\|Ah\|\,\|\tE\|_{L^p(\TT^2)}\,\|g\|_{L^2(M\dD v,L^{q}(\dD x))}\\
 &&\leq&&&K'\,\|Ah\|\,\|\tE\|_{L^p(\TT^2)}\,\|g\|^{1-\eta}\,\|Cg\|^{\eta}
 \end{aligned}
\]
where we have used bounds on $\phiinf$ and H\"older and Sobolev inequalities and noticed that $2/q=1-\eta$. Then, one concludes thanks to bounds on $\|\tE\|_{L^p(\TT^2)}$. 

To estimate $\cR_A$, we simply notice that 
\[
 |\cR_A(h,g,\tg)|\,\leq\,|\cR_0(h,g,\tg)|\,+\,\|A^2h\|\,\|\tE\|_{L^\infty(\TT^2)}\,\|Ag\|.
\]

As for the third estimate, the first term of $\cR_C$ is dealt with similarly using once again the $L^\infty$ bound on $\tE$. However the second term requires a more careful distribution of spatial derivatives, essentially as in the proof of the bound on $\|n\|_{L^p(\T^2)}$ of Proposition~\ref{prop:est-density}. Namely, set $p=2/\eta_2\in(2,\infty)$ and define $q\in(2,\infty)$ by $1/p+1/q=2$, then 
\[
\left|\lla CAh,\, g\,\nabla_x\tE\rra\right|
\,\leq\,K\,\|ACh\|\,\,\|\nabla_x\tE\|_{L^p(\TT^2)}\,\|g\|_{L^2(M\dD v,L^{q}(\T^2))}
\]
with
\[
\begin{aligned}
&\|g\|_{L^2(M\dD v,L^{q}(\T^2))}
&\leq&&&K'\,\left(\int_{\R^2}\|g(\cdot,v)\|_{L^2(\T^2)}^{2\left(1-\frac2p\right)}\,\|Cg(\cdot,v)\|_{L^2(\T^2)}^{\frac4p} M(v)\,\dD v\right)^{\frac12}\\
 &&\leq&&&K''\,\|g\|^{1-\tfrac2p}\,\|Cg\|^{\tfrac2p}\,
\end{aligned}
\]
by Sobolev embeddings, H\"older inequalities and $L^\infty$ bounds on $\phiinf$. The estimate is achieved by relying on bounds on $\|\nabla_x\tE\|_{L^p(\TT^2)}$.

We skip the estimate of $\cR_{AC}$ as it is completely similar.
\end{proof}

\section{Linear warm-up}\label{s:linear}
For expository purpose and to support our choice of exponents in the functional \eqref{func} we first develop our strategy on the following equation
\be\label{VFP}
\d_th + L_\tau h\ =\ 0\,,
\ee
supplemented with initial data $h_0$. Recall that $-L_\tau$  generates a semi-group of contractions on $\cH$ \cite{Helffer-Nier_book, Herau-Nier}. 
Moreover since they form a core for $L_\tau$ it is sufficient to deal with Schwartz initial data, for which all following computations are readily justified.

As already expounded in 
Section~\ref{s:strategy} 
our goal is to prove that in each regime under suitable conditions on parameters $\beta\in\R^3$ and  $\gamma\in(0,+\infty)^3$, the following functional 
\[
\begin{aligned}
&\nrm h \nrm_{t}&=&&&\|h\|^2 
 + \gamma_1\tau^{\beta_1}\min\left(1,\tfrac{t}{\tau}\right)\|Ah\|^2 \\
 &&&+&& \gamma_2\tau^{\beta_2}\min\left(1,\tfrac{t}{\tau}\right)^{3}\|Ch\|^2 
 + 2\gamma_3\tau^{\beta_3}\min\left(1,\tfrac{t}{\tau}\right)^{2}\lla Ah, Ch\rra,
 \end{aligned}
\]
is decaying in time with dissipation rate at least
\[
\begin{aligned}
&D_{t}(h)&=&&&
\tau^{-1}\|Ah\|^2 
+ \gamma_1\tau^{\beta_1-1}\min\left(1,\tfrac{t}{\tau}\right)\|A^*\cdot Ah\|^2\\ 
&&&+&& \gamma_2\tau^{\beta_2-1}\min\left(1,\tfrac{t}{\tau}\right)^{3}\|ACh\|^2 
+ \gamma_3\tau^{\beta_3}\min\left(1,\tfrac{t}{\tau}\right)^{2}\|Ch\|^2.
\end{aligned}
\]

Note that in order to prove exponential decay it is then sufficient to invoke the following entropy/dissipation inequality stemming directly from the Poincaré inequality of Proposition~\ref{prop:est-poincare} and bounds on $\phiinf$. 
\bl
For any $(\beta,\gamma)\in(0,+\infty)^6$ and any $\delta_0>0$ there is a positive constant $K$ such that for any $\tau>0$, any $\delta\in(\delta_0,\infty)$, any $h\in\cH_0$ and any $t\geq \tau$
\[
K\,\min\left(\tau^{-1}, \tau^{\beta_3}, \tau^{\beta_3-\beta_2}\right)\,\left[\|h\|^2 
+\tau^{\beta_1}\,\|Ah\|^2
+\tau^{\beta_2}\,\|Ch\|^2\right]\,\leq\,D_{\gamma,\,\beta,\,\tau,\,t}(h)\,.
\]
\label{lem:poincare2}
\el

\bpr[Diffusive regime] Under the following conditions on $\beta\in\R^3$
 \be
\max\left(1, \frac{\beta_1+\beta_2}{2}\right)\leq \beta_3\leq \min\left(2\beta_1+1, \beta_2-1\right)\,,
\label{condbetadiff}
\ee
for any $\tau_0>0$ there exist $\gamma\in(0,+\infty)^3$, $c_0>0$, $C_0>0$ and $\tilde{\theta}>0$ such that for any $\rho_*\in H^{-1}(\T^2)$ such that $\int_{\T^2}\rho_*=1$, any $\delta>0$ and any $\tau\in(0,\tau_0)$ 
\begin{enumerate}
\item for any $h$, for any $t\geq0$,
$$
\nrm h\nrm_{t}\geq
c_0\left(\|h\|^2 
+\tau^{\beta_1}\min\left(1,\tfrac{t}{\tau}\right)\|Ah\|^2
+\tau^{\beta_2}\min\left(1,\tfrac{t}{\tau}\right)^{3}\|Ch\|^2
\right)
$$
and
$$
\nrm h\nrm_{t}\leq 
C_0\left(\|h\|^2 
+\tau^{\beta_1}\min\left(1,\tfrac{t}{\tau}\right)\|Ah\|^2
+\tau^{\beta_2}\min\left(1,\tfrac{t}{\tau}\right)^{3}\|Ch\|^2
\right)\,;
$$
\item for any $h_0\in\cH_0$ the solution to the linear Vlasov-Fokker-Planck equation \eqref{VFP} starting from $h_0$ satisfies for all $t\geq0$
\[
 \nrm h(t,\cdot,\cdot) \nrm_{t}^2 + \tilde{\theta}\int_0^t\,D_{s}(h(s,\cdot,\cdot))\,\emph{\dD} s\,\leq\,\|h_0\|^2\,.
\]
\end{enumerate}
\label{p:unifestdiff}
\epr
\begin{proof}
One may adapt Lemma~\ref{lem:defQR} to \eqref{VFP} with resulting modifications being that there is no electric field in time derivatives and in remainder terms no trilinear term, no $\tQ_{AC}$ and no $Q_C$. In particular, only $Q_{A}$ and half of $Q_{AC}$ have non-zero contribution to remainders and those may be bounded without resorting to $\phiinf$ bounds. This leads to
 \[
\frac12\nrm h \nrm_{t}^2 + \int_0^t\,D_{s}(h)\,\dD s\,\leq\,\frac 12\|h_0\|^2 + \int_0^t\,\cR_{s}(h)\,\dD s\,,
 \]
with
 \[
\begin{aligned}
  & \cR_{t}(h)&=&&&\gamma_1\tau^{\beta_1}\min\left(1,\tfrac{t}{\tau}\right)\|Ch\|\,\|Ah\|\\
&&+&&&\gamma_3\tau^{\beta_3-1}\min\left(1,\tfrac{t}{\tau}\right)^{2}\left(\|Ah\|\|Ch\|+ 2\|A^2h\|\|A Ch\|\right)\\
&&+&&&\frac12\chi_{t<\tau}\left(\gamma_1\tau^{\beta_1-1}\|Ah\|^2 
+ 3\gamma_2\tau^{\beta_2-1}\left(\tfrac{t}{\tau}\right)^{2}\|Ch\|^2 
+ 4\gamma_3\tau^{\beta_3-1}\left(\tfrac{t}{\tau}\right)\|Ah\|\,\|Ch\|\right)
   \end{aligned}
 \]
where we have used notation $\chi_{t<\tau}$ to denote the value at time $t$ of the characteristic function of $[0,\tau)$, namely $0$ if $t\geq \tau$, $1$ otherwise. 

Now we want to ensure that $\cR_{t}$ is controlled by an arbitrarily small fraction of $D_{t}$, uniformly in time and in $\tau$. 
This is obtained by repeated use of Young's inequality on each remainder terms, carefully dispatching powers of $\tau$ and time weights. To simplify the choice of constants, we seek $\gamma_j$ under the form $\eps^{c_j}$ for some positive $c_j$. As an example let us first show how to control the term coming from $Q_{AC}$. We have
\begin{align*}
\eps^{c_3}\tau^{\beta_3-1}&\min\left(1,\tfrac{t}{\tau}\right)^{2}\,\|Ah\|\|Ch\|\\
 =&\ \eps^{c_3/2}\,\tau^{\beta_3-1 -(\beta_3-1)/2}\min\left(1,\tfrac{t}{\tau}\right)\,\left(\tau^{-1/2}\|Ah\|\right)\,\left(\eps^{c_3/2}\,\tau^{\beta_3/2}\,\min\left(1,\tfrac{t}{\tau}\right)\,\|Ch\|\right)\, \\
 \leq&\ \eps^{c_3/2}\,\tau^{(\beta_3-1)/2}\min\left(1,\tfrac{t}{\tau}\right)\,D_{t}(h)\,,
\end{align*}
and $c_3>0$ and $\beta_3\geq 1$ is required for the right-hand side to be bounded by $\theta_\eps\,D_{t}(h)$, uniformly in time, for small $\tau$, for some given $\theta_\eps\in(0,1)$ going to zero when $\eps$ goes to zero. To repeat the procedure on all terms, we observe that more generally a bound
$$
\eps^{c}\tau^\beta\,\min\left(1,\tfrac{t}{\tau}\right)^\alpha
K\,L\,\stackrel{\eps\to0}{=}\,
o(\eps^{c'}\tau^{\beta'}\,\min\left(1,\tfrac{t}{\tau}\right)^{\alpha'}K^2
+\eps^{c''}\tau^{\beta''}\,\min\left(1,\tfrac{t}{\tau}\right)^{\alpha''}L^2)
$$
uniform with respect to $K$, $L$, $t$ and $\tau\leq\tau_0$ requires $\alpha\geq\min(\{\alpha',\alpha''\})$. If $\alpha'\neq\alpha''$ and $\alpha\in[\alpha',\alpha'']$ then there exists a unique $\theta \in[0,1]$ such that $\alpha=\theta\,\alpha'+(1-\theta)\,\alpha''$ and the estimate also requires $\beta\geq\theta\beta'+(1-\theta)\beta''$. However once the above conditions are fulfilled the estimate holds provided that $c>\theta c'+(1-\theta)c''$. The final upshot is that our claimed estimate followd from 
$$
\begin{array}{ccc}\ds
\beta_1-1\geq-1\,,&\ds
\beta_2-1\geq\beta_3\,,&\ds
\beta_3-1\geq\tfrac12(\beta_3-1)\,,\\
\beta_1\geq\tfrac12(\beta_3-1)\,,&\ds
\beta_3-1\geq \tfrac23(\beta_3-1)-\tfrac13
\,,&\ds
\beta_3-1\geq\tfrac12(\beta_1+\beta_2-2)\,,
\end{array}
$$
--- which reduces to \eqref{condbetadiff} --- and 
$$
\tfrac12(c_1+c_2)<c_3<\min(c_2,2c_1)\,.
$$
We conclude by noting that the choice $(c_1,c_2,c_3)=(1,2,7/4)$ fulfills the latter constraint.

So far we have omitted the very first constraint on $\nrm\cdot\nrm_{t}$. Yet a sufficient condition is 
$$
\beta_3\geq\tfrac12(\beta_1+\beta_2)\,,\quad \gamma_3<\sqrt{\gamma_1\,\gamma_2}\,,
$$
which is redundant with above requirements.
\end{proof}

The proof of Proposition~\ref{p:unifestdiff} may be readily adapted to cope with the regime of evanescent collisions, yielding the following result 
where constraints on $\beta_1$, $\beta_2$ and $\beta_3$ are reversed to ensure uniform bounds for arbitrarily large $\tau$s. 

\bpr[Evanescent collisions] Under the following conditions on $\beta\in\R^3$
\be
\min\left(1, \frac{\beta_1+\beta_2}{2}\right)\geq \beta_3\geq \max\left(2\beta_1+1, \beta_2-1\right)\,.
\label{condbetacoll}
\ee
for any $\tau_0>0$ there exist $\gamma\in(0,+\infty)^3$, $c_0>0$, $C_0>0$ and $\tilde{\theta}>0$ such that for any $\rho_*\in H^{-1}(\T^2)$ such that $\int_{\T^2}\rho_*=1$, any $\delta>0$ and any $\tau\in(\tau_0,+\infty)$ 
\begin{enumerate}
\item for any $h$, for any $t\geq0$,
$$
\nrm h\nrm_{t}\geq
c_0\left(\|h\|^2 
+\tau^{\beta_1}\min\left(1,\tfrac{t}{\tau}\right)\|Ah\|^2
+\tau^{\beta_2}\min\left(1,\tfrac{t}{\tau}\right)^{3}\|Ch\|^2
\right)
$$
and
$$
\nrm h\nrm_{t}\leq 
K_0\left(\|h\|^2 
+\tau^{\beta_1}\min\left(1,\tfrac{t}{\tau}\right)\|Ah\|^2
+\tau^{\beta_2}\min\left(1,\tfrac{t}{\tau}\right)^{3}\|Ch\|^2
\right)\,;
$$
\item for any $h_0\in\cH_0$ the solution to the linear Vlasov-Fokker-Planck equation \eqref{VFP} starting from $h_0$ satisfies for all $t\geq0$
\[
 \nrm h(t,\cdot,\cdot) \nrm_{t}^2 + \tilde{\theta}\int_0^t\,D_{s}(h(s,\cdot,\cdot))\,\emph{\dD} s\,\leq\,\|h_0\|^2\,.
\]
\end{enumerate}
\label{p:unifestcoll}
\epr

\section{Strongly collisional regime}\label{s:strong-collisions}

In view of Lemmas~\ref{lem:estQ} and~\ref{lem:estR}, to analyze terms that have been left over in the foregoing section, we only need to consider
\be
\label{quadrem}
\begin{array}{rclcl}
\ds 
Q^\delta_t(h)&:=&\ds
\tau^{\beta_2}\ \min\left(1,\tfrac{t}{\tau}\right)^3\ \|Ah\|\ \|Ch\|
&+&\ds
\tau^{\beta_2}\ \min\left(1,\tfrac{t}{\tau}\right)^3\ \|h\|\ \|Ch\|\\
&+&\ds
\tau^{\beta_3}\ \min\left(1,\tfrac{t}{\tau}\right)^2\ \|Ah\|^2
&+&\ds
\tau^{\beta_3}\ \min\left(1,\tfrac{t}{\tau}\right)^2\ \|h\|\ \|Ch\|\\
&+&\ds
\tau^{\beta_3}\ \min\left(1,\tfrac{t}{\tau}\right)^2\ \|h\|\ \|Ah\|&&\\
&=:&\ds\frac{1}{\delta^2}\ \sum_{i=1}^5\, S_{i,t}(h)&&
\end{array}
\ee
and choosing some $\eta\in(0,1)^9$ and $\teta_1\in(0,1)$
\be
  \ds \cR^\delta_t(h,g,\tg)\ :=\ \frac{1}{\delta^2}\ \left(\,T_{1,t}(h,g,\tg)\,\chi_{t\geq\tau}\,+\,\tT_{1,t}(h,g,\tg)\,\chi_{t\leq\tau}\,\right) \ +\ \frac{1}{\delta^2}\ \sum_{i=2}^9\, T_{i,t}(h,g,\tg)
  \label{trilinrem}
\ee
where
\[
\begin{array}{rcl}
\ds T_{1,t}(h,g,\tg) &=&\ds\|Ah\|\ \|g\|^{1-\eta_1}\ \|Cg\|^{\eta_1}\ \|\tg\|\,,\\[0.75em]
\ds \tT_{1,t}(h,g,\tg) &=&\ds\|Ah\|\ \|g\|\ \|\tg\|^{1-\teta_1}\ \|C\tg\|^{\teta_1}\,,\\[0.75em]
\end{array}
\]
\[
\begin{array}{rcl}
\ds T_{2,t}(h,g,\tg) &=&\ds\tau^{\beta_1}\ \min\left(1,\tfrac{t}{\tau}\right)\ \|Ah\|\ \|g\|^{1-\eta_2}\ \|Cg\|^{\eta_2}\ \|\tg\|\,,\\[0.75em]
\ds T_{3,t}(h,g,\tg) &=&\ds\tau^{\beta_1}\ \min\left(1,\tfrac{t}{\tau}\right)\ \|A^2h\|\ \|\tg\|^{1-\eta_3}\ \|C\tg\|^{\eta_3}\ \|Ag\|\,,\\[0.75em]
\end{array}
\]
\[
\begin{array}{rcl}
\ds T_{4,t}(h,g,\tg) &=&\ds\tau^{\beta_2}\ \min\left(1,\tfrac{t}{\tau}\right)^3\ \|ACh\|\ \|Cg\|\ \|\tg\|^{1-\eta_4}\ \|C\tg\|^{\eta_4}\,,\\[0.75em]
\ds T_{5,t}(h,g,\tg) &=&\ds\tau^{\beta_2}\ \min\left(1,\tfrac{t}{\tau}\right)^3\ \|ACh\|\ \|\tg\|^{1-\eta_5}\ \|C\tg\|^{\eta_5}\ \|g\|^{\eta_5}\ \|Cg\|^{1-\eta_5}\,,\\[0.75em]
\end{array}
\]
and
\[
\begin{array}{rcl}
\ds T_{6,t}(h,g,\tg) &=&\ds\tau^{\beta_3}\ \min\left(1,\tfrac{t}{\tau}\right)^2\ \|A^2h\|\ \|\tg\|^{1-\eta_6}\|C\tg\|^{\eta_6}\ \|Cg\|\,,\\[0.75em]
\ds T_{7,t}(h,g,\tg) &=&\ds\tau^{\beta_3}\ \min\left(1,\tfrac{t}{\tau}\right)^2\ \|A^2h\|\ \|\tg\|^{1-\eta_7}\ \|C\tg\|^{\eta_7}\ \|g\|^{\eta_7}\ \|Cg\|^{1-\eta_7}\,,\\[0.75em]
\ds T_{8,t}(h,g,\tg) &=&\ds\tau^{\beta_3}\ \min\left(1,\tfrac{t}{\tau}\right)^2\ \|Ch\|\ \|g\|^{1-\eta_8}\ \|Cg\|^{\eta_8}\ \|\tg\|\,,\\[0.75em]
\ds T_{9,t}(h,g,\tg) &=&\ds\tau^{\beta_3}\ \min\left(1,\tfrac{t}{\tau}\right)^2\ \|ACh\|\ \|\tg\|^{1-\eta_9}\|C\tg\|^{\eta_9}\ \|Ag\|\,.\\[0.5em]
\end{array}
\]
We also need to take into account the electric field contributions in time derivatives by augmenting $\nrm \cdot\nrm_{\gamma,\,\beta,\,\tau,\,t}^2$ to  $\cE_{\gamma,\,\beta,\,\tau,\delta,\,t}$ defined by
\[
 \cE_{\gamma,\,\beta,\,\tau,\delta,\,t}(h) = \nrm h \nrm_{\gamma,\,\beta,\,\tau,\,t}^2 + \delta^2(1 + \gamma_1\tau^{\beta_1}\min\left(1,\tfrac{t}{\tau}\right))\|E\|_{L^2}^2\,.
\]
In the following, when dependencies in parameters are not crucial we may drop some indices on this functional by writing  simply  $\cE_t$ in order to lighten notations. 

Prior to carrying on our nonlinear analysis in the diffusive regime, for reading's sake we make a specific choice of $\beta$, namely $\beta=(0,2,1)$, and fix a corresponding $\gamma$ accordingly. This choice is motivated by the following remark.

\br[Optimality of $\beta$]
From Proposition~\ref{p:unifestdiff} 
and Lemma~\ref{lem:poincare2},
one may derive exponential decay of $\nrm h(t,\cdot,\cdot) \nrm_{\gamma,\,\beta,\,\tau,\,t}$ when $h$ solves \eqref{VFP}, explicitly encoded by a rate  $e^{-\theta\tau^{\max(\beta_3,-1,\beta_3-\beta_2)}t}$ for some uniform $\theta>0$. To optimize the former decay rate, one must minimize $\max(\beta_3,-1,\beta_3-\beta_2)$ under constraints \eqref{condbetadiff}. The optimal choice requires actually $\beta_3=1$, that forces $\beta_1=0$ and $\beta_2=2$. Indeed the existence of a $\beta_3$ satisfying constraint~\eqref{condbetadiff} is equivalent to
$$
\beta_1\geq0\,,\quad \beta_2\geq 2\,,\quad 3\beta_1\geq\beta_2-2\,,\quad \beta_1\leq\beta_2-2
$$
which is compatible with $(\beta_1+\beta_2)/2\leq1$ only if $\beta_1=0$ and $\beta_2=2$. In turn, the corresponding choice of $\beta$ does satisfy \eqref{condbetadiff}.
\label{r:optimaldiff}
\er

Last preliminary results are provided by the following proposition. 

\bpr[$\tau\lesssim 1$;\; $\beta=(0,2,1)$]
Set $\beta=(0,2,1)$. For any $\tau_0>0$ and any $\gamma\in(0,\infty)^3$ satisfying corresponding conditions of Proposition~\ref{p:unifestdiff}, for any $\rho_*\in W^{1,p}(\T^2)$, $p>2$, such that $\int_{\T^2}\rho_*=1$, any $\delta_0>0$ and any $\teta_1\in(0,1)$, there exist $\eta\in(0,1)^9$ and $K>0$ such that for any $t\geq0$, any $\tau\in(0,\tau_0)$, any $\delta\in(\delta_0,\infty)$ and any $(h,g,\tg)\in(\cH_0)^3$
\[
\begin{array}{rcl}
\ds Q^\delta_t(h)&\leq&\ds\frac{K}{\delta^2}\ D_{t}(h)\,,\\
&&\\
 \ds \cR^\delta_t(h,g,\tg)&\leq&\ds\frac{K}{\delta^2} \ \left[D_{t}(h)\right]^{\tfrac12}\left[\left(D_{t}(g)\right)^{\tfrac12}\  +\  \tau^{\tfrac12-\teta_1}\,\left(\tfrac{t}{\tau}\right)^{-\frac{3\teta_1}{2}}\,\chi_{t\leq\tau}\,\|g\|\right]\left[\cE_t(\tg)\right]^{1/2}\,.
 \end{array}
\]
\label{p:remdiff}
\epr
\begin{proof}
Besides obvious estimates we point out that Poincar\'e's inequality implies that for any $h\in\cH_0$
$$
\tau^{1/2}\,\min\left(1,\tfrac{t}{\tau}\right)\,\|h\|\,\leq\,K\,\left[D_{t}(h)\right]^{1/2}
$$
uniformly in $\tau\leq\tau_0$, $\delta\geq\delta_0$, $t\geq0$. With this in hand one readily deduce for any $h\in\cH_0$
 \[
 \begin{array}{rcl}
  S_{1,t}(h) + S_{3,t}(h)&\leq& K\,\tau^2\, \min\left(1,\tfrac{t}{\tau}\right)^2\,D_{t}(h)\,,\\[.75em]
  S_{2,t}(h) + S_{5,t}(h)&\leq& K\,\tau\, \min\left(1,\tfrac{t}{\tau}\right)\,D_{t}(h)\,,\\[.75em]
  S_{4,t}(h)&\leq& K\,D_{t}(h)\,,
  \end{array}
 \]
hence proving the first estimate. Likewise one obtains for any $(h,g,\tg)\in(\cH_0)^3$
 \[
 \begin{array}{rcl}
  T_{1,t}(h,g,\tg)&\leq& K\,\min\left(1,\tfrac{t}{\tau}\right)^{-1}\,D_{t}(h)^{1/2}\,D_{t}(g)^{1/2}\,\cE_{t}(\tg)^{1/2}\,,\\[.75em]
  \tT_{1,t}(h,g,\tg)&\leq& K\,\tau^{1-\teta_1}\,D_{t}(h)^{1/2}\,\left[\tau^{-1/2}\min\left(1,\tfrac{t}{\tau}\right)^{-3\teta_1/2}\|g\|\right]\,\cE_{t}(\tg)^{1/2}\,,\\[.75em]
  T_{2,t}(h,g,\tg) + T_{8,t}(h,g,\tg)&\leq& K\,D_{t}(h)^{1/2}\,D_{t}(g)^{1/2}\,\cE_{t}(\tg)^{1/2}\,,
  \end{array}
 \]
 and for $i\in\{3,4,5,6,7,9\}$
 \[
  T_{i,t}(h,g,\tg)\ \leq\  K\,\tau^{1-2\eta_i}\,\min\left(1,\tfrac{t}{\tau}\right)^{\frac{1-3\eta_i}{2}}\,D_{t}(h)^{1/2}\,D_{t}(g)^{1/2}\,\cE_{t}(\tg)^{1/2}\,.
 \]
This yields the second estimate by choosing $\eta_i \in (0,1/3]$ for $i\in\{3,4,5,6,7,9\}$.
\end{proof}

\br[Follow-up on the optimality of $\beta$]
Estimates of Proposition~\ref{p:remdiff} will be used to set up a contraction argument with $\beta=(0,2,1)$ for large enough $\delta$. As pointed out in Remark~\ref{r:optimaldiff}, this choice of $\beta$ is motivated by our will to optimize decay rates. However one may wonder whether with a different choice of $\beta$ one could improve the foregoing estimates and set up a contraction argument using smallness of $\tau$ and not largeness of $\delta$, hence allowing for asymptotically vanishing $\delta$ (possibly in a $\tau$-dependent way). Unfortunately, the answer is negative since our estimate of $S_{4,t}$ is actually independent of $\tau$ and $\beta$.
\label{r:optimaldiff-bis}
\er

\noindent{\bf Proof of Theorem~\ref{t:maindiff}.}
To prove Theorem~\ref{t:maindiff}, we introduce 
\[
X\ =\ \left\{\ h\in L^\infty(\RR_+; \cH)\ \middle|\ \cE(h)<\infty\ \right\}
\qquad\textrm{and}\qquad
Y\ =\ \left\{\ h\in L^\infty(\RR_+; \cH)\ \middle|\ \cF(h)<\infty\ \right\}
\]
endowed with norms $\sqrt{\cE}$ and $\sqrt{\cF}$. For any $R>0$ we denote by $X_R$ and $Y_R$ the closed balls of center $0$ and radius $R$ of Banach spaces $X$ and $Y$. We recall that $\cE$ and $\cF$ are defined by 
$$
\cF(h)\ =\ \cE(h) + \theta\,D(h)
$$
where
$$
\cE(h)\,=\,\sup_{t\geq0}\cE_t(h(t,\cdot,\cdot))
\qquad\qquad\textrm{and}\qquad\qquad
D(h)\,=\,\int_0^\infty D_t(h(t,\cdot,\cdot))\,\dD t\,.
$$

We fix $R_0>0$ and choose $h_0\in\cH_0$ such that $\|h_0\|\leq R_0$ and for a suitable $R>0$ we consider the map $\Phi\,:\,X_R\to L_{loc}^\infty(\RR_+; \cH),$ $\tg\,\mapsto h$ where $h$ starts from $h_0$ and solves the linear equation
\be
\partial_th \,+\, L_\tau h \,-\, E\cdot v\ =\ \tilde{E} \cdot A^*h
\label{linVPFP}
\ee
where $E$ and $\tE$ are obtained through the Poisson equation from respectively $h$ and $\tg$. Existence and uniqueness in $\mathcal{C}(\R_+; \cH)$  for \eqref{linVPFP} may be shown for instance using arguments \cite[Proposition 5.1]{Herau_FP-confining} (adapted to our space-periodic setting that does not involve a confining potential) in two steps. First, when considered as given source terms, $E\cdot v$ and $\tilde{E} \cdot A^*h$ satisfy the hypotheses of \cite[Proposition 5.1]{Herau_FP-confining} thanks to estimates
\[
 \int_0^T\left|\lla E\cdot v, \varphi\rra\right|\,\dD t\ \leq\ \frac{K}{\delta^2}\|h\|_{L^\infty(0,T; \cH)}\,\|A\varphi\|_{L^2(0,T; \cH)}\,\sqrt{T}
\]
and
\[
  \int_0^T\left|\lla\tilde{E} \cdot A^*h, \varphi\rra\right|\,\dD t\ \leq\ \frac{K\,R}{\delta^2}\|h\|_{L^\infty(0,T; \cH)}\,\|A\varphi\|_{L^2(0,T; \cH)}\,\left(\int_0^T\min(1,\tfrac{t}{\tau})^{-3\eta}\right)^{1/2}
\]
for any $\varphi\in L^2(0,T;\cH)$ such that $A\varphi\in L^2(0,T;\cH)$ and $\eta>0$  small enough. 
Note that the latter estimate is obtained in the same way $\cR_0$ was bounded in Lemma~\ref{lem:estR}. 
Then, by a fixed point argument in $\mathcal{C}(0,t_0; \cH)$ for a sufficiently small $t_0$, one builds a unique solution to \eqref{linVPFP}. Since $t_0$ can be chosen independently  of the initial data one may repeat the argument to eventually get a global solution.

Our goal is to show that when $\delta$ is large enough one may choose $R$ sufficiently large (independently of $\delta$) such that $\Phi(X_R)\subset Y_R$ and $\Phi$ is a strict contraction (with uniform constant) for norms $\sqrt{\cE}$ and $\sqrt{\cF}$. Since fixed points of $\Phi$ are exactly solutions of \eqref{VPFPabstract} starting from $h_0$, this will prove altogether the existence of a solution in $Y_R$, its uniqueness in $X_R$, uniform bounds on the solution and smooth dependence on $h_0$. To extend the uniqueness result one shall only need to remark that the above argument may be localized in time and to use a continuity argument based on the fact that any solution belongs to a suitable time-localized version of $X_R$ for sufficiently small time.

Let us be more precise on the order in which parameters are chosen. Positive parameters $\tau_0$ and $R_0$ are given data and we choose a first $\delta_0>0$ arbitrarily, say $\delta_0=1$. Then we may set $\beta=(0,2,1)$ and a suitable $\gamma$ is provided by Proposition~\ref{p:unifestdiff}, constants in corresponding estimates being uniform in the range $\tau\leq\tau_0$, $\delta\geq\delta_0$. The parameter $\theta$ could be chosen essentially arbitrarily but it is convenient to set $\theta=\ttheta/2$ where $\ttheta$ is provided by Proposition~\ref{p:unifestdiff}. It turns out that we may also choose $R>R_0$ arbitrarily, say $R=2R_0$. 

\smallskip

\emph{Step 1, $\Phi(X_R)\subset Y_R$}. Now we show that we may choose $\delta_0'\geq\delta_0$ such that for any $\delta\geq\delta'_0$, $\tau\leq\tau_0$ and $\|h_0\|\leq R_0$, we do have that $\Phi(X_R)\subset Y_R$.
Combining Propositions~\ref{p:unifestdiff} and~\ref{p:remdiff}, we obtain indeed that 
 for any $(\tau,\delta,h_0)$ as above, for any $\tg\in X_R$, $h=\Phi(\tg)$ satisfies for any $t\geq0$, for some constants $K'$ and $K$ depending only on $\teta_1\in(0,1)$
\[
\begin{array}{rcl}
\ds
\cE_t(h(t))&+&\ds
\ttheta \int_0^t D_s(h(s))\,\dD s\\
&\leq&\ds
\|h_0\|^2\,+\,
K'\,\int_0^t\left[Q_s^\delta(h(s))\,+\,R_s^\delta(h(s),h(s),\tg(s))\right]\,\dD s\\[.75em]
&\leq&\ds
\|h_0\|^2\,+\,
\frac{K\,(1+R)}{\delta^2}\left(\int_0^t D_s(h(s))\,\dD s\,+\,\int_0^{\min(t,\tau)}
\tau^{1-2\teta_1}\,\left(\tfrac{s}{\tau}\right)^{-3\teta_1}
\|h(s)\|^2\,\dD s\right)
\end{array}
\]
hence, provided that $(1+R)/(\delta_0')^2$ is sufficiently small one has
\[
 \cE_t(h(t))\,+\,\ds
\theta \int_0^t D_s(h(s))\,\dD s\ \leq\ \|h_0\|^2\,+\,
\frac{K\,(1+R)}{\delta^2}\,\int_0^{\min(t,\tau)}
\tau^{1-2\teta_1}\,\left(\tfrac{s}{\tau}\right)^{-3\teta_1}
\|h(s)\|^2\dD s\,,
\]
and choosing $\teta_1\in(0,\tfrac13)$ yields for any $t\geq0$ and some constant $K$ depending only on $\teta_1$
$$
\|h(t)\|^2
\,\leq\,
\|h_0\|^2\,\eD^{\frac{K\,(1+R)\,\tau^{2(1-\teta_1)}}{\delta^2}}
$$
therefore for any $t\geq0$ and some constant $K$ depending only on $\teta_1$
$$
\cF(h)\,\leq\,\|h_0\|^2\,\left(1+\frac{K\,(1+R)\,\tau^{2(1-\teta_1)}}{\delta^2}\eD^{\frac{K\,(1+R)\,\tau^{2(1-\teta_1)}}{\delta^2}}\right)
$$
which can be made smaller than $R^2$ provided that $(1+R)/(\delta_0')^2$ is small enough. It follows that $\Phi$ is well-defined from $X_R$ to $Y_R$.

\smallskip

\emph{Step 2, Contraction}. Now we show that $\Phi$ is a strict contraction from $X_R$ to $Y_R$. Provided that $(1+R)/(\delta_0')$ is sufficiently small, for any $(\tau,\delta,h_0)$ as above, for any data $(\tg_1,\tg_2)\in (X_R)^2$, values $h_1=\Phi(\tg_1)$ and $h_2=\Phi(\tg_2)$ satisfy 
\[
 (\partial_t + L_\tau)(h_1-h_2)\ +\ (E_1-E_2)\cdot v\ =\ (\tE_1-\tE_2)\cdot A^*h_1 + \tE_2\cdot A^*(h_1 - h_2)
\]
(with obvious implicit notation for electric fields), thus, for some constant $K$, for any $t\geq0$,
$$
\cE_t((h_1-h_2)(t))\ +\ \tilde{\theta}\,\int_0^tD_{s}((h_1-h_2)(s))\;\dD s\ 
\leq\ \frac{K\,(1+R)}{\delta^2}\,\cF(h_1-h_2)^{\tfrac12}\,\cE(\tg_1-\tg_2)^{\tfrac12}
$$
as may be derived using that $h_1\in Y_R$ and $h_2\in Y_R$. Factoring out $\cF(h_1-h_2)^{\tfrac12}$, it follows that $\Phi$ is Lipschitz from $X_R$ to $Y_R$ and that its Lipshitz constant may be assumed arbitrarily small provided that $(1+R)/(\delta_0')^2$ is sufficiently small.
This is sufficient to lead to the well-posedness part of Theorem~\ref{t:maindiff}

\smallskip

\emph{Step 3, Exponential Decay}. The large-time decay may then be deduced from the Poincar\'e inequality. Indeed the foregoing arguments provide for any $t_2\geq t_1\geq \tau$
$$
\cE_{t_2}(h(t_2))\,+\,\theta\,\int_{t_1}^{t_2}\,D_s(h(s))\dD s\,\leq\,\cE_{t_1}(h(t_1))\,,
$$
and thanks to Lemma~\ref{lem:poincare2}, for some $K>0$, our choice of parameters yields uniformly for $(\tau,\delta,h_0)$ as above that for all $t\geq\tau$
\[
 K\,\tau\,\cE_{t}(h(t))\leq D_t(h(t))\,.
\]
Thanks to a backward Grönwall-like argument, this leads to 
$$
\cE_t(h(t))\,\leq\,\eD^{-\theta'\,\tau\,t}\cE_\tau(h(\tau))
\,\leq\,K'\,\eD^{-\theta'\,\tau\,t}\|h_0\|
$$
for any $t\geq\tau$ and some uniform positive $K'$ and $\theta'$. This may be extended to all $t$ using the uniform boundedness of $\cE_t(h(t))$. Similar arguments prove the uniform stability with respect to initial data.

\section{The regime of evanescent collisions}\label{s:evanescent}

\br[Optimality of $\beta$]
From Proposition~\ref{p:unifestcoll} and Lemma~\ref{lem:poincare2}, one may exponential decay in $\cH$ of solutions to \eqref{VFP}, explicitly encoded by a rate  $e^{-\theta\tau^{\min(\beta_3,-1, \beta_3-\beta_2)}t}$ for some uniform $\theta>0$. In order to optimize the former decay rate under \eqref{condbetacoll}, first observe that the latter constraint implies $\beta_3-\beta_2\geq-1$ hence $\min(\beta_3,-1, \beta_3-\beta_2)=\min(\beta_3,-1)$. Thus, we only need to ensure that
$\beta_3\geq -1$ and there is a large choice of $\beta$ that meet this constraint jointly with \eqref{condbetacoll}. Indeed for any $-2\leq\beta_1\leq0$ one may choose $\beta_2$ such that
$$
\max\left(3\beta_1+2,-\beta_1-2\right)\leq \beta_2\leq \beta_1+2
$$
and then a suitable $\beta_3$ may be chosen according to
$$
\min\left(1, \frac{\beta_1+\beta_2}{2}\right)\geq \beta_3\geq \max\left(-1,2\beta_1+1, \beta_2-1\right)
$$
and that defines a non empty interval of $\beta_3$s.
\label{r:optimalcoll}
\er

A natural guide towards a good choice of $\beta$ could be the examination of the best analogue of Proposition~\ref{p:remdiff} in the regime where $\tau$ is large. However, another thing that also differs from the strongly collisional regime is that the main obstruction here does not arise from quadratic terms. Indeed one may prove the following estimates.

\bpr[$\tau\gtrsim 1$;\; $\beta=(-1,-1,-1)$]
Set $\beta=(-1,-1,-1)$. For any $\tau_0>0$ and any $\gamma\in(0,\infty)^3$ satisfying corresponding conditions of Proposition~\ref{p:unifestcoll}, for any $\rho_*\in W^{1,p}(\T^2)$, $p>2$, such that $\int_{\T^2}\rho_*=1$, any $\delta_0>0$, there exists $K>0$ such that for any $t\geq0$, any $\tau\in(\tau_0,\infty)$, any $\delta\in(\delta_0,\infty)$ and any $h\in\cH_0$
\[
Q^\delta_t(h)\,\leq\,\frac{K}{\delta^2}\ D_{t}(h)\,.
\]
\label{p:remcolllin}
\epr

Since we shall not make any use of the former proposition we skip its proof. Yet let us point out that the involved uniform estimate enforces $\beta_j\leq-1$, $j=1$, $2$, $3$. Indeed constraints~\eqref{condbetacoll} implies $\beta_1\leq0$ and, under this condition, the best possible estimates
$$
\begin{array}{rcl}\ds
S_{1,t}(h)&\leq&\ds
K\,\tau^{\frac{2\beta_2+1-\beta_3}{2}}\, \min\left(1,\tfrac{t}{\tau}\right)^2\,D_{t}(h)\,,\\[.5em]\ds
S_{3,t}(h)&\leq&\ds
K\,\tau^{1+\beta_3}\, \min\left(1,\tfrac{t}{\tau}\right)^2\,D_{t}(h)\,,
\end{array}
$$
provide uniform bounds only when $2\beta_2+1\leq\beta_3$ and $\beta_3\leq -1$, which jointly with \eqref{condbetacoll} yield the claimed constraint. Note also that if moreover one requires $\beta_3\geq -1$ then the only possible choice is indeed $\beta=(-1,-1,-1)$.

Unfortunately, in the regime $\tau\gtrsim1$, trilinear terms leads to a more stringent constraint on $\delta$ and the foregoing choice $\beta=(-1,-1,-1)$ does not minimize trilinear constraints. Indeed
\[
\begin{array}{rcl}
T_{1,t}(h,g,\tg)&\leq& K\,\tau^{1-\tfrac{\eta_1}{2}\,(\beta_3+1)}\min\left(1,\tfrac{t}{\tau}\right)^{-1}\,D_{t}(h)^{1/2}\,D_{t}(g)^{1/2}\,\cE_{t}(\tg)^{1/2}\,,\\[.75em]
\tT_{1,t}(h,g,\tg)&\leq&K\,\tau^{\frac{1+\max(1,-\beta_3)-\teta_1\beta_2}{2}}\,D_{t}(h)^{1/2}\,\left[\tau^{-1/2}\min\left(1,\tfrac{t}{\tau}\right)^{-3\teta_1/2}\|g(t)\|\right]\,\cE_{t}(\tg)^{1/2}\,,
\end{array}
\]
with $0<\eta_1\leq1$ and $0<\teta_1\leq1$, provides bounds that grow superlinearly in $\tau$ unless $\beta_2\geq0$ and $\beta_3\geq-1$. Note in turn that in order not to exceed a linear growth in $\tau$, the bound on $S_{3,t}$ only requires $\beta_3\leq0$. Now let us observe that in the final argument $\teta_1$ is constrained by $\teta_1<\tfrac13$ and that constraints~\eqref{condbetacoll} yield 
$$
\tfrac32\beta_3+\tfrac12\,\leq\,\beta_2\,\leq\,\beta_3+1\,.
$$
In turn minimizing 
$$
\max(1-\tfrac16\beta_2,\beta_2-\tfrac12\beta_3+\tfrac12)
$$ 
under these constraints proves that one cannot do better than a $\tau^{\tfrac{14}{15}}$-growth and that one may hope to (almost) realize it only with 
$$
\beta=(-\frac{8}{15},\frac25,-\frac{1}{15})\,.
$$
As the following proposition proves this turns out to be indeed possible.

\bpr[$\tau\gtrsim 1$;\; $\beta=(-\tfrac{8}{15},\tfrac25,-\tfrac{1}{15})$]
Set $\beta=(-\tfrac{8}{15},\tfrac25,-\tfrac{1}{15})$. For any $\eps>0$, any $\tau_0>0$ and any $\gamma\in(0,\infty)^3$ satisfying corresponding conditions of Proposition~\ref{p:unifestcoll}, for any $\rho_*\in W^{1,p}(\T^2)$, $p>2$, such that $\int_{\T^2}\rho_*=1$ and any $\delta_0>0$, there exist $\eta\in(0,1)^9$, $\teta_1\in(0,1)$ and $K>0$ such that for any $t\geq0$, any $\tau\in(\tau_0,+\infty)$, any $\delta\in(\delta_0,\infty)$ and any $(h,g,\tg)\in(\cH_0)^3$
\[
\begin{array}{rcl}
\ds Q^\delta_t(h)&\leq&\ds\frac{K\,\tau^{\tfrac{14}{15}}}{\delta^2}\ D_{t}(h)\,,\\
&&\\
 \ds \cR^\delta_t(h,g,\tg)&\leq&\ds\frac{K\,\tau^{\tfrac{14}{15}+\eps}}{\delta^2} \ \left[D_{t}(h)\right]^{\tfrac12}\left[\left(D_{t}(g)\right)^{\tfrac12}\  +\  \tau^{-\tfrac12}\,\left(\tfrac{t}{\tau}\right)^{-\frac{3\teta_1}{2}}\,\chi_{t\leq\tau}\,\|g\|\right]\left[\cE_t(\tg)\right]^{1/2}\,.
 \end{array}
\]
\label{p:remcoll}
\epr

\begin{proof}
We have already shown how to bound $S_{1,t}$ and $S_{3,t}$. Moreover $S_{2,t}$ may be bounded as $S_{1,t}$, and $S_{4,t}$ and $S_{5,t}$ as $S_{3,t}$. Likewise  
$$
S_{4,t}(h)\,\leq\,
K\,\tau^{\tfrac12(1+\beta_3)}\,D_{t}(h)=K\,\tau^{\tfrac{7}{15}}\,D_{t}(h)\,.
$$
We have already explained how to bound $\tT_{1,t}$, we only need to add that $\teta_1$ is chosen as $\teta_1=\tfrac13-5\eps$ if $\eps<\tfrac{1}{15}$, and arbitrarily otherwise. Besides
\[
\begin{array}{rcl}
T_{1,t}(h,g,\tg)&\leq& K\,\tau^{1-\tfrac{\eta_1}{2}\,(\beta_3+1)}\min\left(1,\tfrac{t}{\tau}\right)^{-1}\,D_{t}(h)^{1/2}\,D_{t}(g)^{1/2}\,\cE_{t}(\tg)^{1/2}\,,\\[.75em]
T_{2,t}(h,g,\tg)&\leq& K\,\tau^{\beta_3-\tfrac{\eta_2}{2}(\beta_3+1)}\,D_{t}(h)^{1/2}\,D_{t}(g)^{1/2}\,\cE_{t}(\tg)^{1/2}\,,\\[.75em]
T_{8,t}(h,g,\tg)&\leq& K\,\tau^{\tfrac{1-\eta_8}{2}(\beta_3+1)}\,D_{t}(h)^{1/2}\,D_{t}(g)^{1/2}\,\cE_{t}(\tg)^{1/2}\,,
\end{array}
\]
which are shown to be sufficient by choosing $\eta_1=\tfrac17$ and any $\eta_2$, $\eta_8$. At last
 \[
\begin{array}{rcl}
  T_{3,t}(h,g,\tg)&\leq&K\,\tau^{\tfrac12\beta_1-\tfrac{\eta_3}{2}\beta_2+1}\,\min\left(1,\tfrac{t}{\tau}\right)^{\frac{1-3\eta_3}{2}}\,D_{t}(h)^{1/2}\,D_{t}(g)^{1/2}\,\cE_{t}(\tg)^{1/2}\,,\\[.75em]
  T_{4,t}(h,g,\tg)&\leq&K\,\tau^{\tfrac{1-\eta_4}{2}\beta_2-\tfrac12\beta_3+\tfrac12}\,\min\left(1,\tfrac{t}{\tau}\right)^{\frac{1-3\eta_4}{2}}\,D_{t}(h)^{1/2}\,D_{t}(g)^{1/2}\,\cE_{t}(\tg)^{1/2}\,,\\[.75em]
  T_{5,t}(h,g,\tg)&\leq&K\,\tau^{\tfrac12\beta_2-\tfrac12\beta_3+\tfrac12+\tfrac{\eta_5}{2}(1+\beta_3-\beta_2)}\,\min\left(1,\tfrac{t}{\tau}\right)^{\frac{1-3\eta_5}{2}}\,D_{t}(h)^{1/2}\,D_{t}(g)^{1/2}\,\cE_{t}(\tg)^{1/2}\,,\\[.75em]
  T_{6,t}(h,g,\tg)&\leq&K\,\tau^{-\tfrac12\beta_1-\tfrac{\eta_6}{2}\beta_2+\tfrac12\beta_3+\tfrac12}\,\min\left(1,\tfrac{t}{\tau}\right)^{\frac{1-3\eta_6}{2}}\,D_{t}(h)^{1/2}\,D_{t}(g)^{1/2}\,\cE_{t}(\tg)^{1/2}\,,\\[.75em]
  T_{7,t}(h,g,\tg)&\leq&K\,\tau^{-\tfrac12\beta_1+\tfrac12\beta_3+\tfrac12+\tfrac{\eta_7}{2}(1+\beta_3-\beta_2)}\,\min\left(1,\tfrac{t}{\tau}\right)^{\frac{1-3\eta_7}{2}}\,D_{t}(h)^{1/2}\,D_{t}(g)^{1/2}\,\cE_{t}(\tg)^{1/2}\,,\\[.75em]
  T_{9,t}(h,g,\tg)&\leq&K\,\tau^{-\tfrac{(1+\eta_9)}{2}\beta_2+\beta_3+1}\,\min\left(1,\tfrac{t}{\tau}\right)^{\frac{1-3\eta_9}{2}}\,D_{t}(h)^{1/2}\,D_{t}(g)^{1/2}\,\cE_{t}(\tg)^{1/2}\,,
\end{array}
 \]
which are themselves shown to be sufficient by choosing $\eta_i$, $i\in\{3,4,5,6,7,9\}$, arbitrarily in $(0,\tfrac13]$.
\end{proof}

With this in hand the proof of Theorem~\ref{t:maincoll} follows as in Theorem~\ref{t:maindiff}.

\section{Asymptotic models in the diffusive regime}\label{s:asymptotics}

This section is devoted to the proof of Theorem~\ref{t:asymptotics}. So far we have aimed at global-in-time estimates and therefore what exactly was the reference time scale was immaterial. Now we turn to asymptotics that are uniform only locally in time thus we explicitly introduce a reference time in the equations. Namely, after choosing an observation time $\tref=\tref(\tau)$ in a $\tau$-dependent way, we observe that if $f$ solves the original system then $\fref$ defined by $\fref(t,x,v)=f(\tref\,t,x,v)$ and $\href=(\fref-\finf)/\finf$ are such that
\be
\left\{
\begin{array}{l}
\ds \frac{1}{\tref(\tau)}\d_{t}\href \,+\,L_\tau\href\,=\,\Eref\cdot A^*(\funcun)\,+\,\Eref\cdot A^*\href,\\
\,\\
\ds \Eref\ =\ \frac{1}{\delta^2}\nabla_{x}\Delta_x^{-1}\nref\,,
\qquad \nref\ =\ \int_{\R^3}\,\href\,\finf\,\dD{v}.
\end{array}
\right.
\label{VPFPscaled}
\ee
with initial data $h_0$.

We already know that for any $\tau_0>0$, $R_0>0$ there exists $\delta_0$ and uniform positive constants $K$ and $\theta$ such that when $\delta>\delta_0$ and $\|h_0\|\leq R_0$, for any $t\geq0$
$$
\begin{array}{rcl}\ds
\|\href(t)\|^2&+&\ds
\min\left(1,\tfrac{\tref}{\tau}t\right)\,\|A\href(t)\|^2
\,+\,\tau^2\min\left(1,\tfrac{\tref}{\tau}t\right)^3\,\|C\href(t)\|^2\\[0.5em]
&+&\ds
\frac{\tref}{\tau}\int_0^t\,\|A\href(s)\|^2\dD s
\qquad+\qquad
\tref\,\tau\int_0^t\min\left(1,\tfrac{\tref}{\tau}s\right)^2\,\|C\href(s)\|^2\dD s\\[0.5em]
&+&\ds
\frac{\tref}{\tau}\int_0^t\min\left(1,\tfrac{\tref}{\tau}s\right)\,\|A^2\href(s)\|^2\dD s
\,+\,
\tref\,\tau\int_0^t\min\left(1,\tfrac{\tref}{\tau}s\right)^3\,\|AC\href(s)\|^2\dD s
\\[0.5em]
&\leq&\ds
K\,\|h_0\|^2
\end{array}
$$
and
$$
\|\href(t)\|\,\leq\,K\,\eD^{-\theta\,\tref\,\tau\,t}\|h_0\|\,.
$$
The latter estimates shows that $\|\href(t)\|$ converges to $0$ uniformly on compacts of $(0,+\infty]$ provided that $\tref(\tau)\tau\stackrel{\tau\to0}{\to}\infty$. The foregoing asymptotic regime is stationary. This proves part $(i)$ of Theorem~\ref{t:asymptotics}.

\subsection{Asymptotically linear free-field regimes}

In the opposite regime where $\tref(\tau)\tau\stackrel{\tau\to0}{\to}0$ we show now that relevant asymptotic models are of evolution type or at least strongly keep track of initial data. However they are linear and one can also drop out convective terms at least in the velocity directions.

Namely, let us denote $\hlin$ the solution to 
$$
\frac{1}{\tref(\tau)}\d_{t}\hlin \,+\,v\cdot\nabla_x\hlin\,+\,\Einf\cdot\nabla_v\hlin\,+\,\frac{1}{\tau}A^*A\hlin\,=\,0
$$ 
with initial data $h_0$. Observe that for any $t\geq0$
$$
\|\hlin(t)\|^2
\,+\,2\,\frac{\tref}{\tau}\int_0^t\,\|A\hlin(s)\|^2\dD s
\,\leq\,\|h_0\|^2\,.
$$

Note that 
$$
\begin{array}{rcl}\ds
\frac{1}{\tref(\tau)}\d_{t}(\href-\hlin)&+&\ds
v\cdot\nabla_x(\href-\hlin)+\Einf\cdot\nabla_v(\href-\hlin)
+\,\frac{1}{\tau}A^*A(\href-\hlin)\\[0.5em]
&=&
\Eref\cdot A^*\href
\,+\,\Eref\cdot A^*(\funcun)\,.
\end{array}
$$
This implies for any $t\geq0$
$$
\begin{array}{rcl}\ds
\|(\href-\hlin)(t)\|^2
&+&\ds
\frac{\tref}{\tau}\,\int_0^t\|A(\href-\hlin)(s)\|^2\dD s\\[0.5em]
&\leq&\ds
\,\tref\,\tau\int_0^t\left(\|\Eref(s)\href(s)\|^2
\,+\,\|\Eref(s)\|^2\right)\,\dD s\\[0.5em]
&\leq&\ds
\frac{K_\eta\,(1+R_0^2)}{\delta^2}
\|h_0\|^2
\times\begin{cases}
\,(\tref\,\tau\,t)^{1-\eta}&\qquad\textrm{when }\quad \tref\,t\geq\tau\\
\,\tau^{2(1-\eta)}&\qquad\textrm{when }\quad \tref\,t\leq\tau
\end{cases}\,\,.
\end{array}
$$
for any $0<\eta<\tfrac13$. The trickiest part of the foregoing bound follows from
$$
\begin{array}{rcl}\ds
\tref\,\tau\int_0^t\|\Eref(s)\href(s)\|^2\,\dD s
&\leq&\ds
\frac{K_\eta}{\delta^2}
\tref\,\tau\int_0^t\|C\href(s)\|^{2\eta}\|\href(s)\|^{2(1-\eta)}\|\href(s)\|^{2}\,\dD s\\[0.5em]
&\leq&\ds
\frac{K_\eta\,R_0^{2(1-\eta)}}{\delta^2}\|h_0\|^2
\tref\,\tau\int_0^t\tfrac{\min\left(1,\tfrac{\tref}{\tau}s\right)^{2\eta}\|C\href(s)\|^{2\eta}}{\min\left(1,\tfrac{\tref}{\tau}s\right)^{2\eta}}\,\dD s\\[0.5em]
&\leq&\ds
\frac{K_\eta\,R_0^2}{\delta^2}\|h_0\|^2
(\tref\,\tau)^{1-\eta}\left(\int_0^t{\min\left(1,\tfrac{\tref}{\tau}s\right)^{-\frac{2\eta}{1-\eta}}}\,\dD s\right)^{1-\eta}\,\,.
\end{array}
$$
obtained by the same argument used to bound $\cR_0$ in Lemma~\ref{lem:estR} and H\"older estimates. The proof of the claim is then achieved by noticing that 
$$
\int_0^t{\min\left(1,\tfrac{\tref}{\tau}s\right)^{-\frac{2\eta}{1-\eta}}}\,\dD s
\,\leq\, K\,\max\left(t,\tfrac{\tau}{\tref}\right)\,.
$$
Hence we are asymptotically close to the linear regime. 

Now let us show that contributions of the linear field terms also vanish in these regimes. Let us denote $\has$ the solution to
$$
\frac{1}{\tref(\tau)}\d_{t}\has \,+\,\frac{1}{\tau}A^*A\has\,=\,0\,,
$$ 
with initial data $h_0\in\cH$, and assume $Ch_0\in\cH$. Note that 
$$
\|\has(t)\|^2
\,+\,2\,\frac{\tref}{\tau}\int_0^t\,\|A\has(s)\|^2\dD s
\,\leq\,\|h_0\|^2\,.
$$
and that, since $C$ and $A$ commute, $C\has$ solves the same equation, hence satisfies a similar estimate with initial data $Ch_0$. Using that $v=A+A^*$, a direct estimate provides
$$
\begin{array}{rcl}\ds
\|(\hlin-\has)(t)\|^2
&+&\ds
\frac{\tref}{\tau}\int_0^t\,\|A(\hlin-\has)(s)\|^2\dD s\\[0.5em]
&\leq&\ds
K\,\tref\,\tau\int_0^t\,\|C\has(s)\|^2\dD s\\[0.5em]
&+&\ds
K\tref\int_0^t\,\left[\|AC\has(s)\|+\|A\has(s)\|\right]\,\|(\hlin-\has)(s)\|\dD s\\[0.5em]
&\leq&\ds
K'\left[\tref\,\tau\,t\,\|Ch_0\|^2+\sqrt{\tref\,\tau\,t}\left(\|h_0\|^2+\|Ch_0\|^2\right)\right]\,.
\end{array}
$$
Since such $h_0$ form a dense set in $\cH$, the corresponding convergence, uniform in $t$ on compact sets of $[0,\infty)$, 
$$
\left\|\hlin(t)-\has(t)\right\|
\stackrel{\tau\to0}{\longrightarrow}0
\qquad\textrm{provided }\qquad \tref(\tau)\,\tau\stackrel{\tau\to0}{\to}0
$$
may be extended to any $h_0\in \cH$. This proves part $(iv)$ of Theorem~\ref{t:asymptotics}.

At last observe that, on one hand, for any $t\geq0$, with $n_0 = \int_{\RR^2}h_0\,M\dD v$
$$
\|\has(t)- n_0\|\,\leq\,\eD^{-\theta\,\frac{\tref}{\tau}\,t}\,\|h_0-n_0\|
$$
for some uniform $\theta>0$, which proves part $(iii)$ of Theorem~\ref{t:asymptotics}. On the other hand, uniformly in $t$ on compact sets of $[0,\infty)$, 
$$
\left\|\has(t)-h_0\right\|
\stackrel{\tau\to0}{\longrightarrow}0
\qquad\textrm{provided }\qquad \tfrac{\tref(\tau)}{\tau}\stackrel{\tau\to0}{\to}0
$$
for any $h_0\in \cH$. The latter follows through a density argument from the explicit estimate
$$
\left\|\has(t)-h_0\right\|
\,\leq\,
\tfrac{\tref(\tau)\,t}{\tau}\,\|A^*A\,h_0\|
$$
that holds when moreover $A^2h_0\in \cH$. This proves part $(v)$ of Theorem~\ref{t:asymptotics}.

\subsection{Nonlinear diffusive regime}\label{s:nonlinasymp}

The remaining regime corresponds to the case where $\tref\,\tau$ is of order $1$ and therefore in this section we set $\tref(\tau)=\tau^{-1}$. Note that with this choice we already know that
$$
\left(\int_0^\infty\|A\href(t)\|^2\,\dD t\right)^{\tfrac12}
\,\leq\,K\,\tau\,\|h_0\|
$$
and this implies that
$$
\left(\int_0^\infty\|\fref(t)-M\rhoref(t)\|_{L^2(M^{-1})}^2\dD t\right)^{\tfrac12}
\,\leq\,K'\,\tau\,\|h_0\|\,.
$$

Our goal is to also identify some asymptotic limiting behavior for $\nref = \rhoref - e^{-\phiinf}$. Note that the proof of Theorem~\ref{t:maindiff} by a contraction argument also provides us with the fact the map $h_0\mapsto \href$ is Lipschitz from $\cH_0$ to $L^\infty(\R_+;\cH_0)$, therefore this is also the case for the map $h_0\mapsto \nref$ from $\cH_0$ to $L^\infty(\R_+;L^2(\T^2))$. To prove convergence in $L_{loc}^\infty(\R_+;L^2(\T^2))$ (without explicit decay rates) we may therefore restrict to a case where also hold $Ah_0\in\cH$, $ACh_0\in\cH$, $Ch_0\in\cH$ and $C^2h_0\in\cH$. The gain we shall use is two-fold. Indeed one may both drop out time weights in our arguments and upgrade it to higher regularity to obtain
$$
\begin{array}{rcl}\ds
\|\href(t)\|^2&+&\ds
\,\|A\href(t)\|^2
\,+\,\|C\href(t)\|^2
\,+\,\|AC\href(t)\|^2
\,+\,\tau^2\,\|C^2\href(t)\|^2\\[0.5em]
&+&\ds
\frac{1}{\tau^2}\int_0^t\,\|A\href(s)\|^2\dD s
\quad+\quad
\frac{1}{\tau^2}\int_0^t\,\|A^2\href(s)\|^2\dD s\\[0.5em]
&+&\ds
\frac{1}{\tau^2}\int_0^t\,\|AC\href(s)\|^2\dD s
\quad+\quad\frac{1}{\tau^2}\int_0^t\,\|A^2C\href(s)\|^2\dD s\\[0.5em]
&+&\ds
\int_0^t\,\|AC^2\href(s)\|^2\dD s\quad+\quad\int_0^t\,\|C\href(s)\|^2\dD s
\quad+\quad\int_0^t\,\|C^2\href(s)\|^2\dD s\\[0.75em]
&\leq&\ds
K\,\left[\|h_0\|^2\,+\,\|Ah_0\|^2\,+\,\|Ch_0\|^2\,+\,\|ACh_0\|^2\,+\,\tau^2\,\|C^2h_0\|^2\right]\,.
\end{array}
$$
It is very important to note however that in order to do so we do not need to restrict further $\delta_0$ in a way that would depend on the size of $\|Ah_0\|$, $\|Ch_0\|$, $\|ACh_0\|$ and $\|C^2h_0\|$. Otherwise this would prevent us from extending the convergence to $h_0\in\cH_0$ by a density argument. In contrast, we will be free to restrict $\tau_0$ in a way depending on above norms. We skip the proof of the foregoing claim as lengthier but otherwise completely similar to estimates that have been proved in detail above.

Now, our starting point is the continuity equation~\eqref{e:n}, that here takes the form
$$
\d_t \nref \,+\,\frac{1}{\tau}\textrm{div}_x(\jref)\,=\,0
$$
where we recall that 
$$
\nref\:=\int_{\R^2}\href\,\finf \dD v\,,\qquad \jref\,:=\,\int_{\R^2} v\,\href\,\finf \dD v\,.
$$
Similarly, a momentum equation 
\be\label{e:j}
\tau\d_t\jref+\frac{1}{\tau}\jref\,=\,\nref\,\Einf\,+\,(\rhoinf+\nref)\,\Eref-\nabla_x\nref-\textrm{div}_x(\Sref)
\ee
may be obtained by multiplying \eqref{VPFPscaled} by $v$ and integrating in the velocity variable, with
$$
\Sref\,=\,
\int_{\R^2} (v\otimes v-\I)\,\href\,\finf \dD v\,.
$$
To give some hints on computations involved in the foregoing derivation, we introduce $\langle\,\cdot\,;\,\cdot\,\rangle_v$ to denote the spatially dependent scalar product on $L^2(\finf\,\dD v)$ and notice that 
$$
\nref\,=\,\langle \funcun ; \href \rangle_v\,,\qquad \jref\,=\, \langle A^*(\funcun) ; \href \rangle_v\,=\,\langle \funcun ; A\href \rangle_v\,,\qquad
\Sref\,=\,\langle (A^*)^2(\funcun);\href \rangle_v
$$
as follows from $v=A+A^*$, $A(\funcun)=0$ and commutation properties of $A$ and $A^*$. Now the key computations leading to the above are
$$
\begin{array}{rcl}\ds
\langle A^*(\funcun); \Eref\cdot A^*(\funcun+\href)\rangle_v
&=&\ds
(\rhoinf+\nref)\,\Eref\,,\\[0.5em]\ds
\langle A^*(\funcun); \Einf\cdot A\href\rangle_v
&=&\ds
\Sref\Einf\,,\\[0.5em]\ds
\langle A^*(\funcun); A^*\cdot A\href\rangle_v
&=&\ds
\jref\,,\\[0.5em]\ds
\langle A^*(\funcun); \textrm{div}_x ((A+A^*)\href)\rangle_v
&=&\ds
\textrm{div}_x (\langle A^*(\funcun); (A+A^*)\href\rangle_v)
\,-\,\langle A^*(\funcun);(A+A^*)\href\rangle_v\,\Einf\\[0.5em]\ds
&=&\ds
\textrm{div}_x(\Sref\,+\,\nref\,\I)-\,(\Sref\,+\,\nref\,\I)\,\Einf\,.
\end{array}
$$

To proceed, with this in hands, the continuity equation may be turned into
$$
\d_t\nref\,+\,\textrm{div}_x(\nref\,\Einf\,+\,(\rhoinf+\nref)\,\Eref-\nabla_x\nref)
\,=\,\textrm{div}_x(\textrm{div}_x(\Sref))\,+\,\tau\textrm{div}_x(\d_t\jref)\,.
$$
Therefore we set
$$
\tnref\,=\,\nref-\tau\textrm{div}_x(\jref)
\qquad\textrm{and}\qquad
\tEref\,=\,\frac{1}{\delta^2}\nabla_x\Delta_x^{-1}\tnref
$$
and introduce $\nas$ the solution to 
$$
\d_t\nas\,+\,\textrm{div}_x(\nas\,\Einf\,+\,(\rhoinf+\nas)\,\Eas-\nabla_x\nas)
\,=\,0
$$
starting from $n_0$, where $\Eas=\delta^{-2}\nabla_x\Delta_x^{-1}\nas$. Note that the well-posedness of the equation for $\nas$ may be obtained by a simpler version of the argument proving Theorem~\ref{t:maindiff}. Moreover we may ensure that $\|\tn_0\|\leq 2\,R_0$ by requiring $\tau_0\,\|\textrm{div}_x(j_0)\|\leq R_0$ and thus restrict $\delta_0$ in a way that depends only on $R_0$ in order to deduce for any $t\geq0$
$$
\begin{array}{rcl}\ds
\|(\tnref-\nas)(t)\|^2&+&\ds
\int_0^t\|\nabla_x(\tnref-\nas)(s)\|^2\dD s\\[0.75em]\ds
&\leq&\ds
\tau^2\,\|\textrm{div}_x(j_0)\|^2\\[0.5em]
&+&\ds
K\int_0^t\Big[\,\tau^2\|\textrm{div}_x(\jref)(s)\|_{L^4(\T^2)}^2\,\|(\Einf+\tEref(s))\|_{L^4(\T^2)}^2\\[0.5em]
&&\ds
\quad\quad\,+\,\|\nref(s)\|^2\,\|(\tEref-\Eref)(s)\|_{L^\infty(\T^2)}^2\\[0.5em]
&&\ds
\quad\quad
\,+\,\tau^2\|\nabla^2(\jref)(s)\|^2\,+\,\|\textrm{div}_x(\Sref)(s)\|^2\,\Big]\,\dD s\\[0.75em]
&\leq&\ds
K'\,\tau^2\,\left[\|h_0\|^2\,+\,\|Ah_0\|^2\,+\,\|Ch_0\|^2\,+\,\|ACh_0\|^2\,+\,\tau^2\,\|C^2h_0\|^2\right]
\end{array}
$$
since $\|\textrm{div}_x(\Sref)\|\leq\,K\,\left(\|A^2C\href\|+\|A^2\href\|\right)$. This proves part $(ii)$ of Theorem~\ref{t:asymptotics}.

\bibliographystyle{abbrv}
\bibliography{bibli} 
\end{document}